\documentclass[]{article}

\usepackage[hmargin=3cm,vmargin=3.5cm]{geometry}

\usepackage{makeidx}  

\usepackage{amsthm} 

\usepackage{latexsym}
\usepackage[pdftex]{graphics}
\usepackage{xypic}
\usepackage{amsmath}
\usepackage{amssymb}
\xyoption{all} 
\usepackage{mathrsfs}
\usepackage{eufrak}

\newtheorem{theorem}{Theorem}[section]

\newtheorem{lemma}[theorem]{Lemma}

\newtheorem{definition}[theorem]{Definition}

\newtheorem{proposition}[theorem]{Proposition}

\newtheorem{corol}[theorem]{Corollary}

\newtheorem{remark}[theorem]{Remark}

\newcommand{\PCM}{\mbox{\bf PCM}}

\newcommand{\Comp}{{\circ}}
\newcommand{\Nat}{{\mathbb N}}
\newcommand{\CC}{{\mathcal C}}
\newcommand{\D}{{\mathcal D}}
\newcommand{\E}{{\mathcal E}}
\newcommand{\F}{{\mathcal F}}
\newcommand{\HH}{{\mathcal H}}
\newcommand{\K}{{\mathcal K}}
\newcommand{\CS}{\mbox{\bf Cat}_\Sigma}
\newcommand{\CAT}{\mbox{\bf cCat}}
\newcommand{\cat}{\mbox{\bf Cat}}

\title{A categorical analogue of the monoid semiring construction} 
\author{Peter M. Hines}

\begin{document}

\maketitle

\begin{abstract}
This paper introduces and studies a categorical analogue of the familiar monoid semiring construction. By introducing an axiomatisation of summation that unifies  notions of summation from algebraic program semantics with various notions of summation from the theory of analysis, we demonstrate that the monoid semiring construction generalises to cases where both the monoid and the semiring are categories.  This construction has many interesting and natural categorical properties, and natural computational interpretations.

\end{abstract}

\section{Introduction}\label{intro}
The monoid semiring construction -- in particular, the special case of the group ring construction --  is one of the most familiar and useful algebraic constructions. This paper places this construction within a significantly more general categorical setting, where the monoid is generalised to a category (with a suitable smallness condition) and the semiring is replaced by a category equipped with a suitable notion of partial summation on hom-sets.

\subsection{Cauchy products and monoid semirings}
In the formal theory of power series, an infinite power series over some complex variable $z$, given as 
$P=\alpha_0 + \alpha_1z+\alpha_2z^2+\ldots$ 
may be treated as simply a function $P:{\mathbb N}\rightarrow {\mathbb C}$. Given another formal power series $Q:{\Bbb N}\rightarrow {\Bbb C}$ over the same variable, their {\em convolution}, or {\em Cauchy product}, is the formal power series $(Q*P):{\Bbb N}\rightarrow {\Bbb C}$ given by $(Q*P)(n) = \sum_{n=y+x} q(y)p(x)$.

A formal power series $P:{\mathbb N}\rightarrow {\mathbb C}$ converges absolutely within the unit disk $\{ \| z\| \leq 1\}$ when the sum $\sum_{n\in {\Bbb N}} P(n)$ converges absolutely and it is a straightforward result of analysis \cite{T} that the Cauchy product of two formal power series that converge absolutely within the unit disk itself converges absolutely within the unit disk (and much more general conditions may also lead to convergence --- we again refer to \cite{T}).

When restricting formal power series to the case where only a finite number of coefficients are non-zero (i.e. polynomials in some complex variable  $z$), convergence is guaranteed not only within the unit disk, but for all $z\in {\Bbb C}$. Algebraically, this naturally generalises to the familiar theory of monoid semirings \cite{JG}.
\begin{definition}{\em Monoid semirings}\label{msr}\\
Let $(M,\cdot)$ be a monoid, and $(R,\times,+,1_R,0_R)$ be a unital semiring. The {\bf monoid semiring} $R[M]$ is the unital semiring whose elements are functions 
 $\eta :M\rightarrow R$ such that $\| \{m: \eta(m)\neq 0  \}_{m\in M} \|<\infty$.
The multiplication and addition in this semiring are given by 
\begin{itemize}
\item $(\eta \times \mu)(m)  = \sum_{m=qp} \eta(q)\mu(p)$, 
\item $(\eta + \mu)(m) = \eta(m)+\mu(m)$.
 \end{itemize}
The additive identity is the function $0(m)=0_R \ \forall m\in M$, and the multiplicative identity is the function $1(m)= \left\{ \begin{array}{lr} 1_R & m=1_M\  \\ 0_R & \mbox{otherwise.} \end{array}\right.$
When $(M,\cdot)$ is a group, $R[M]$ is called the {\bf group semiring}; similarly, when $R$ is a ring, then $R[M]$ is called the {\bf monoid (or group) ring}.
\end{definition}
This paper generalises the above construction of monoid semirings in two ways: 
\begin{itemize}
\item The finite sums of a semiring are replaced by a more general axiomatic summation of (possibly infinite) indexed families.
\item The monoids $(M,\cdot)$ and $(R,\times )$ are replaced by categories. Thus, the unital ring $R$ is replaced by a category with some appropriate notion of summation on homsets. 
\end{itemize}

\section{An axiomatic notion of summation}\label{PCM-start}
For the programme outlined above, we replace semirings with categories equipped with a partial summation on hom-sets. The overall intention of this paper is to provide a categorical analogue of the monoid semiring construction that generalises the usual theory, but is also applicable to categories used in algebraic program semantics.  As discussed in Appendix \ref{appendix}, axiomatisations of summation commonly used within algebraic program semantics have properties that rule out various analytic notions of summation such as absolute convergence of real or complex sums.

We therefore introduce a very general axiomatisation of  summation that includes, as special cases, various notions of summation from both theoretical computer science (in particular, algebraic program semantics) and analysis. By comparison with other notions of summation discussed in Appendix \ref{appendix}, this is a very weak axiomatisation --- in particular, the expressive power we require comes from both the axioms we now present, and the axioms for the interaction of summation and composition given in Section \ref{summcats}.

\begin{definition}\label{PCM}{\em Partial Commutative Monoids (PCMs)}\\ 
Given sets  $M$ and $I$, an {\bf $I$-indexed family of elements} of $M$  is defined to be a function $x:I\rightarrow M$. For simplicity, we denote this by $\{ x_i \}_{i\in I}$. Throughout this paper, we restrict ourselves to countable (i.e. either finite or denumerably infinite) indexing sets, and hence   {\bf countably indexed families}.

A {\bf partial commutative monoid} or {\bf PCM} is  a non-empty set $M$ together with a  
partial function $\Sigma$ from indexed families of $M$ to elements of $M$. An indexed family of elements of $M$ is called {\bf summable} when it is in the domain of $\Sigma$, and summation is required to satisfy the following two axioms:
\begin{enumerate}
\item The {\bf Unary Sum axiom} Any family $\{ x_i \}_{i\in I}$, where 
$I = \{i'\}$ is a singleton
set, is summable, and $\sum_{i\in I}  x_i  = x_{i'}$.
\item The {\bf Weak Partition-Associativity axiom} Let $\{x_i\}_{i\in I}$ be a 
{\em summable} family, and let $\{I_j\}_{j\in J}$ be a countable partition\footnote{Following \cite{MA}, we also allow countably many $I_j$ to be empty.} of $I$. Then 
$\{ x_i\}_{i\in I_j}$ is summable for every $j\in J$, as is $\{ \sum_{i\in I_j}x_i
\}_{j\in J}$, and 
\[ \sum_{i\in I} x_i = \sum_{j\in J} \left(\sum_{i\in I_j} x_i \right) \]
\end{enumerate}  
$ $ \\
Given a summable family $x = \{x_i\}_{i\in I}$, we may  write
$\Sigma (x)$  (unambiguously) as $\sum_{i\in I}x_i$.  In particular, 
if $I = \{1, \ldots , n\}$, we  write $\Sigma (x) = x_1 + x_2 + x_3 +
\cdots + x_n$, and if $I = \Nat$, $\Sigma (x)   =  x_1 + x_2 + x_3 +
\cdots $. Notice that by Weak Partition Associativity, we may equate
different partitions of a summable family $x$, for example:
\begin{eqnarray*}
x_1 + x_2 + x_3 + \cdots  & = & x_1 + (x_2 + x_3 + \cdots + x_n + \cdots)\\
& = & (x_1 + x_2) + (x_3 + x_4) + \cdots + (x_n + x_{n+1}) + \cdots \\
& = & (x_1+x_3 +x_5 + \cdots) + (x_2+x_4 + x_6 + \cdots) \\
\end{eqnarray*}
\end{definition}

\begin{remark}{\bf The WPA axiom}\label{hilleQuote}
The above Weak Partition Associativity axiom is clearly a weakening of the usual Partition Associativity Axiom from algebraic program semantics \cite{MA,HA}, where the two-sided implication is weakened to a one-sided version (see Appendix \ref{appendix} for more details). However, in this weakened form it is also familiar from traditional analysis. For example, in \cite{EiHi}, p. 108, the following property of absolute convergence of real number is stated and proved: 
\begin{quotation}
``An absolutely convergent series may be split into mutually exclusive subseries, finite or infinite in number. The sum of these subseries is equal to the sum of the original series.''
\end{quotation}
(Note that the prior assumption of an absolutely convergent series in this quotation means that this statement is not equivalent to the usual Partition Associativity axiom described in  Definition \ref{sigmamon}).
In Appendix \ref{appendix} we give various examples of PCMs from both analysis and algebraic program semantics, and compare this formalism to other axiomatisations of summation used in various fields. 
\end{remark}

We first demonstrate that the indexed summation of a PCM is invariant under isomorphism of indexing sets.

\begin{proposition}\label{welldefined}
Let  $x:I\rightarrow M$ and $y:J\rightarrow M$ be countably indexed families and let $\phi:I\rightarrow J$  be a bijection satisfying $y\Comp \phi = x$. Then 
$\Sigma(y)$
is defined exactly when $\Sigma(x)$ is defined, in which case they are equal.  
\end{proposition}
\begin{proof}   
For arbitrary $i\in I$, the set  $J_i = \{\phi(i)\}$ is a singleton, and hence, by the unary sum axiom, is summable. As indexing families are countable,  $\{J_i \ | \ i\in I\}$ is a countable partition of $J$.  Let us now assume that $\{y_j\} _{j\in J}$ is summable: we deduce that:
\begin{eqnarray*}
\sum_{j\in J} y_j  & =  & \sum_{i\in I}\sum_{j\in J_i} y_j \qquad
\mbox{Weak partition associativity}\\
& = & \sum_{i\in I}y_{\phi(i)} \qquad \mbox{Defn, and Unary sum axiom}\\
& = & \sum_{i\in I}x_i  \qquad \mbox{Defn}
\end{eqnarray*} 
Alternatively, let us assume that $\{y_j\}_{j\in J}$ is not summable. Then the assumption that $\{x_i\}_{i\in I}$ is summable will (by interchanging the roles of $x$ and $y$ in the above argument) imply the summability of $ \{y_j\}_{j\in J}$ -- a contradiction. Therefore, $\{y_j\}_{j\in J}$ is summable exactly when $\{x_i\}_{i\in I}$ is summable, in which case their sums are equal.
\end{proof}

Note the similarity of this notion with either the permutation-independence of {\em absolute convergence} of real sums (Definition \ref{realsums}), or the notion of {\em unconditional convergence} of sums in Banach space (Definition \ref{banachsums}). \\

\noindent The following basic properties of PCMs will be used throughout:
\begin{proposition}\label{basics} Let $(M,\Sigma )$ be a PCM. Then 
\begin{enumerate}
\item {\em (Summable Subfamilies)} Let $\{ x_i \}_{i\in I}$ be a summable family of
$M$. Then any subfamily $\{x_i\}_{i\in K}$,
where $K \subseteq I$, is also summable.
\item {\em (Existence of Zero)} The empty set is summable, and $x+ \{ \} = x = \{ \} +
x$ for all $x\in M$. Hence it is a zero for $M$, and we write $0 = \sum \{ \}$.
\item {\em (Sums of Zeros)} For any index set $I$, let $0_I: I\rightarrow M$
denote the constantly zero family (so $0_I(i) = 0$, for all $i\in I$).
Then $0_I$ is summable, and $\Sigma_I 0_I = 0$.  More generally, for
any element $x \in M$, $x + 0 + 0 + 0 + \cdots  = x ~  $ (where $0 + 0 + \cdots$ 
denotes (the sum of) either a finite or  infinite sequence of $0$'s). 
\end{enumerate}
\label{pcmsums}
\end{proposition}
\begin{proof} {\em The  proofs of (1) and (2) below are based on very similar proofs
(for  the  special case of partially additive monoids -- see Appendix A)
presented in \cite{MA}.}\\
\begin{enumerate}
\item {\em (Summable Subfamilies)} Any subset $K\subseteq I$ defines a partition of $I$, namely  
$\{K, I\setminus K\}$.   By Weak
Partition Associativity,  $\sum_{i\in K}x_i$ exists. 
\item {\em (Existence of Zero)} As $M$ is by definition non-empty, the unary sum axiom
implies that the set of summable families is also non-empty. The empty family is a 
subfamily of any summable family; hence letting $K = \emptyset$ in the partition above,
we see that the empty family
$\{ \}$ is summable. It is then immediate that $\sum \{\} =0$ is a zero for the
summation operation, and $0+x=x=x+0$ exists for arbitrary $x\in M$.
\item {\em (Sums of Zeros)} Pick any partition of $I$ whose first cell is $I$
itself, and the remaining cells are empty (the number of empty cells is either
finite or infinite,  depending upon whether one wishes a    finite (resp. infinite) sum
of
$0$'s.  For example, write $I = I_1
\uplus (\uplus_{n>1} I_n)$, where
$I_1 = I$,
$I_i =
\emptyset$, if $i> 1$.  
  If $x =
\{x_i\}_{i\in I}$ is an
$I$-indexed {\em summable} family, then by Weak Partition Associativity we have:
$\sum_{i\in I}x_i =
\sum_{i\in I_1}x_i +
\sum_{n>1}(\sum_{i\in I_n}x_i) = \sum_{i\in I_1}x_i + 0 + 0 + \cdots$.
Now pick a singleton family $\{x\}$,  so $\Sigma(x) = x$.  The result
follows. 
\end{enumerate}

\end{proof}

\noindent
We now define homomorphisms of PCMs, and show that the class of all PCMs, together with this notion of homomorphism, forms a category:
\begin{definition}{\em PCM homomorphisms, the category of PCMs}\\
A {\bf homomorphism} of PCMs is a function $f:(M,\Sigma )\rightarrow (N,\Sigma')$ satisfying the following natural property: \\

Given a summable family $\{ m_i \}_{i\in I}$ of $(M,\Sigma )$, then $\{ f(m_i) \}_{i\in I}$ is a summable family of $(N,\Sigma')$, and $f\left( \Sigma_{i\in I} m_i  \right) \ = \ \Sigma'_{i\in I} f(m_i)$. \\
\end{definition}

\begin{proposition}
The class of all PCMs, together with the above notion of homomorphism, forms a category that we denote $\PCM$.
\end{proposition}
\begin{proof}
First note that for a PCM $(M,\Sigma^M)$ the identity function $1_M:M\rightarrow M$ is a PCM homomorphism. Next, consider PCM homomorphisms $f:(A,\Sigma^A)\rightarrow (B,\Sigma^B)$ and $g:(B,\Sigma^B)\rightarrow (C,\Sigma^C)$, together with a summable family $\{ a_i\}_{i\in I}$ of $A$.  Then the function $gf:A\rightarrow C$ satisfies $g\left(f\left(\Sigma^A_{i\in I} a_i\right)\right) = g\left( \Sigma^B_{i\in I} f(a_i) \right) = \Sigma^C_{i\in I} gf(a_i)$. 
(Note these sums are required to exist, by the definition of PCM homomorphism).  Thus $gf$ is a PCM homomorphism from $(A,\Sigma^A)$ to $(C,\Sigma^C)$. Finally, associativity of composition follows from the associativity of composition for functions.
\end{proof}
Examples of PCMs are given in Appendix \ref{appendix}.  For the program outlined in Section \ref{intro}, we now require categories whose hom-sets are PCMs, together with a specified interaction between summation and composition.

\section{Categories with a notion of summation on hom-sets}\label{summcats}
We now introduce a certain class of categories whose hom-sets are PCMs, together with axioms for the interaction of summation and composition:

\begin{definition}\label{PCMcats}{\em PCM-categories}\\
We define a {\bf PCM-category} to be a locally small\footnote{i.e. we allow for a proper class of objects, but require that all homsets are indeed sets.} category $\CC$, together with, for all $X,Y\in Ob(\CC)$, a partial function $\Sigma^{(X,Y)}$ from countably indexed families over $\CC(X,Y)$ to $\CC(X,Y)$ (we will often omit the superscript, when this is clear from the context). 

This class of partial functions is required to satisfy the following axioms:
\begin{enumerate}
\item {\bf (PCM-structure on hom-sets)}\\
$\left( \CC(X,Y),\Sigma^{(X,Y)}\right) $ is a PCM, for all $X,Y\in Ob(\CC)$.
\item {\bf (Strong distributivity)}\\
Given summable families $\{ f_i\in \CC(X,Y) \}_{i\in I}$ and $\{ g_j \in \CC(Y,Z)\}_{j\in J}$, then $\{ g_jf_i \in \CC(X,Z)\}_{(j,i)\in J\times I}$ is a summable family satisfying
\[ \sum_{(j,i)\in J\times I }^{\ \ \ \ \ \ \ \ \ (X,Z)} g_jf_i \ = \ \left( \sum_{j\in J}^{\ \ \ \ \ \ \ \ \ (Y,Z)} g_j \right) \left( \sum_{i\in I}^{\ \ \ \ \ \ \ \ \ (X,Y)} f_i \right) \]
\end{enumerate}
\end{definition}
We consider examples of PCM-categories in Appendix \ref{appendix}, and properties of PCM-categories in Section \ref{PCMcatprops} below.

\paragraph*{\bf PCM-categories, and categorical enrichment} A very natural question at this point is whether a ``PCM-category'' is a category enriched over some suitable (monoidal, or closed) category of PCMs. We refer to Section \ref{neverending} for this question.\\

\subsection{Properties of PCM-categories}\label{PCMcatprops}
As may be expected, the strong distributivity property, together with the unary sum axiom, implies the usual left- and right- distributivity laws:
\begin{proposition}\label{distrib}
Let $\left( \CC , \Sigma^{(\underline{ \ } , \underline{ \ })} \right)$ be a PCM category, and let $\{ g_i \in \CC(Y,Z) \}_{i\in I}$ be a summable family. Then, for all arrows $f\in \CC(X,Y)$ and $h\in \CC(Z,T)$, 
\[ \{ hg_i \in \CC(Y,T) \}_{i\in I} \ \mbox{ and } \ \{ g_if \in \CC(X,Z) \}_{i\in I} \]
are summable families, and 
\[ h\left( \sum_{i\in I} g_i \right) \ =\  \sum_{i\in I} (hg_i) \ \ \mbox{ and } \ \ \left( \sum_{i\in I} g_i \right)   f \ = \ \sum_{i\in I} (g_if) \]
\end{proposition}
\begin{proof}
Consider the index set $A=\{ a'\}$, and the indexed family $\{ h_a \}_{a\in A}$ given by $h_{a'}=h$. By the  unitary sum axiom $h=\sum_{a\in A} h_a$, and so 
\[ h\sum_{i\in I} g_i = \left( \sum_{a\in A} h_a \right) \left( \sum_{j\in J}  g_j \right) \]
By strong distributivity 
\[ \left( \sum_{a\in A} h_a \right) \left( \sum_{j\in J}  g_j \right)  = \sum_{(a,j) \in A\times J } h_a g_j \]
As $A$ is a single element set, $A \times J \cong J$, and $h_a=h$. Therefore, by Proposition \ref{welldefined}, 
\[ h\sum_{j\in J} g_j  \ = \ \sum_{j \in J } h g_j \]
The proof for the opposite distributive law is almost identical.
\end{proof}

\begin{corol}\label{zeroarrows}
Every PCM-category has zero arrows.
\end{corol}
\begin{proof}
Let  $\left( \CC , \Sigma^{(\underline{\ },\underline{\ })} \right)$ be a PCM-category. For all $X,Y\in Ob(\CC)$ we define $0_{XY}$ to be the sum of the empty set $ \{ \}_{XY} \subseteq \CC(X,Y)$. Then by the above distributive laws, for all $f\in \CC(Y,Z)$, $f0_{XY} = f\left(\sum\{ \}_{XY}\right)$, and hence $f0_{XY}=\left(\sum \{\}_{XZ}\right)=0_{XZ}$. Similarly $0_{XY}g=0_{WY}$, for all $g\in \CC(W,X)$.
\end{proof}

\begin{remark}{The usual treatment of distributivity}
The usual approach in algebraic program semantics is to take the above left- and right- distributivity laws as  axiomatic, and use the (much stronger) notion of summation to prove an analogue of strong distributivity. This is described in Appendix \ref{appendix}.  We do not take this approach, because the axiomatisation of summation this requires is too strong  --- it imposes the {\em positivity property} that $x+y=0 \ \Rightarrow \ x=0=y$. Were we to have taken this approach, it would have meant ruling out many of the motivating examples for the Cauchy product construction, including Group Rings, convergent polynomials over real and complex variables, etc. 

Instead, as we demonstrate by example in Appendix \ref{appendix}, neither the PCM axiomatisation, nor strong distributivity, implies  the positivity property.\\
\end{remark}

We now consider some implications of strong distributivity:

\begin{proposition}\label{reorderingsummations}
Let $\left( \CC , \Sigma^{(\underline{\ },\underline{\ })} \right)$ be a PCM-cat. and let $\{ g_j \in \CC(Y,Z) \}_{j\in J}$ and $\{ f_i \in \CC(X,Y) \}_{i\in I}$ be summable families. Then $\sum_{j\in J} \left( \sum_{i\in I} g_jf_i \right)$ and $\sum_{i\in I}\left( \sum_{j\in J} g_jf_i\right)$ are both defined, and 
\[ \sum_{i\in I} \left( \sum_{j\in J} g_jf_i \right) = \sum_{(i,j)\in I \times J} g_jf_i  = \sum_{j\in J} \left( \sum_{i\in I} g_jf_i \right) \]
\end{proposition}
\begin{proof}
By the strong distributivity property, the family $\{ g_jf_i \in \CC(X,Z) \}_{(j,i)\in J\times I}$ is summable. Now consider the partition of $J\times I$ given by $\{ \{ (j,i) \}_{j\in J }\}_{i\in I}$. By the weak partition-associativity axiom, for arbitrary fixed $i\in I$ the family $\{ g_jf_i \}_{j\in J}$ is summable, as is $\left\{ \sum_{j\in J} g_jf_i \right\}_{i\in I}$ and 
\[ \sum_{i\in I} \left( \sum_{j\in J} g_jf_i \right) = \sum_{(i,j)\in I \times J} g_jf_i \]
The dual identity
\[ \sum_{j\in J} \left( \sum_{i\in I} g_jf_i \right) = \sum_{(i,j)\in I \times J} g_jf_i \]
follows by partitioning $J\times I$ as $\{ \{  (j,i) \}_{i\in I} \}_{j\in J}$, and therefore 
\[ \sum_{i\in I} \left( \sum_{j\in J} g_jf_i \right) = \sum_{(i,j)\in I \times J} g_jf_i  = \sum_{j\in J} \left( \sum_{i\in I} g_jf_i \right) \]
\end{proof}

\begin{proposition}\label{composingsums}
Let $\left( \CC , \Sigma^{(\underline{\ },\underline{\ })} \right)$ be a PCM-cat. and let $\{ s_i \in \CC(X,X) \}_{i\in I}$ be a summable family. Then for all $n>0$, the family 
\[ \{ s_{i_n}s_{i_{n-1}} \ldots s_{i_2}s_{i_1} \in \CC(X,X) \}_{(i_n,\ldots i_1)\in I^n } \]
is summable, as are all its subfamilies.
\end{proposition}
\begin{proof}
{\em (By induction) }  The result is trivially true for $n=1$. Now assume it holds for some $k>0$. Then 
by strong distributivity, 
\[ \{ s_i s_{i_k} \ldots s_{i_1} \in \CC(X,X) \}_{(i,(i_k,\ldots i_1))\in I\times I^k } \]
is also summable, and our result follows by induction. Finally, recall the summable subfamilies property (Proposition \ref{basics}).
\end{proof}

\begin{corol}\label{monoidsums}
Let $\left( \CC , \Sigma^{(\underline{\ },\underline{\ })} \right)$ be a PCM-cat. and let $F= \{ f_i \in \CC(X,X) \}_{i\in I}$ be a summable family containing the identity. Then
\begin{enumerate} 
\item Arbitrary finite subsets of the submonoid of $\CC(X,X)$ generated by $F$ are summable.
\item Let $F'$ denote the indexed subfamily given by removing all occurrences of $1_X$ from $F$. When there exists some word $w$ in the subsemigroup generated by $F'$ satisfying $w=1_X$, then 
\begin{enumerate}
\item The sum $\sum_{i=1}^M 1_X$ exists, for all $M\in {\Bbb N}$.
\item For all $f\in \CC(X,Y)$ and $g\in \CC(W,X)$, the sums 
\[ \sum_{i=1}^M f \in \CC(X,Y) \ \mbox{ and } \ \sum_{i=1}^M g\in \CC(W,X) \]
exist, for all $M\in {\Bbb N}$.
\end{enumerate}
\end{enumerate}
\end{corol}
\begin{proof}
\begin{enumerate}
\item Consider a finite subset $T\subseteq F^*\subseteq \CC(X,X)$, where $F^*$ denotes the free monoid on $F$. As $T$ is finite, there exists some $K\in {\Bbb N}$ such that each $t\in T$ may be written as a distinct word of no more than $K$ members of $F$. However, since $F$  contains the identity, each word of $T$ may be written as a distinct word of exactly $K$ members of $\{ f_i\}_{i\in I}$. Thus, our result follows by Proposition \ref{composingsums} above and the summable subfamilies property (Proposition \ref{basics}).
\item We now assume the additional condition on $F$ given above:
\begin{enumerate}
\item Let us write $w=1_X$ as a word of $K$ elements of $F'$. Then by Proposition \ref{composingsums} above, the family
\[ \{  1_X^{K(M-N)} w 1_X^{KN}   \}_{N=1..M}  \]
is summable. However, $1_X^{K(M-N)} w 1_X^{KN}=1_X$, for all $N=1..M$. Therefore $\sum_{N=1}^M 1_X$ exists, as does $\sum_{N=1}^{M'}1_X$, for all $0<M'<M$, by the summable subfamilies property (Proposition \ref{basics}).
\item By distributivity (Proposition \ref{distrib} above) $\{  f1_X \ \in \CC(X,Y) \}_{i=1..M}$ exists, and hence $\sum_{i=1}^M f$ exists. The proof for arbitrary $g\in \CC(W,X)$ is similar.
\end{enumerate}
\end{enumerate}
\end{proof}

\subsection{The category of PCM-categories}\label{PCM-end}
The class of all PCM-categories is itself a category:
\begin{definition}{\em PCM-functors, the category $\CS$}\\
Given PCM-categories $\CC , \D$, we say that a functor $\Gamma : \CC\rightarrow \D$ is a {\bf PCM-functor} when:
\begin{itemize}
\item Given a summable family $\{ f_i \in \CC(X,Y) \}_{i\in I}$, then $\{ \Gamma(f_i) \in \D(\Gamma(X),\Gamma(Y)) \}$ is a summable family, and 
\[ \Gamma  \left( \sum_{i\in I} f_i \right)  = \sum_{i\in I} \Gamma (f_i)  \]
\end{itemize}
We denote the category of all PCM-categories and PCM-functors by $\CS$. 
\end{definition}

\begin{proposition}
$\CS$ is well-defined.
\end{proposition}
\begin{proof}
First note that identity functors on PCM-categories are trivially PCM-functors. To prove compositionality, consider two PCM-functors $\Gamma\in \CS (\CC,\D)$ and $\Delta \in \CS(\D,\E)$. By definition, for any summable family $\{ f_i\in \CC(X,Y) \}_{i\in I}$, the family $\{ \Gamma(f_i ) \in \D(\Gamma(X),\Gamma(Y)) \}_{i\in I}$ is summable, as is $\{ \Delta \Gamma (f_i) \in \E (\Delta \Gamma (X) ,\Delta \Gamma (Y)) \}_{i\in I}$. Then, also by definition of PCM-functors, 
\[ \Delta \left( \Gamma \left( \sum_{i\in I} f_i \right) \right) \ = \ 
\Delta \left( \sum_{i\in I} \Gamma (f_i) \right) = \sum_{i\in I} \Delta\Gamma (f_i) \]
and hence $\Delta\Gamma$ is also a PCM-functor.
Finally,  associativity follows from the usual associative property for functors, and thus $\CS$ is well-defined. 
\end{proof}

We also have finite products of PCM-categories:
\begin{proposition}\label{finprod}
The category $\CS$ has finite products.
\end{proposition}
\begin{proof}
Consider $\CC , \D \in Ob(\CS)$. We define their product $\CC\times \D$ in a similar way to the usual product of categories: objects are pairs $(A,X)$, where $A\in Ob(\CC)$ and $X\in Ob(\D)$. The homset $(\CC\times \D)((A,X),(B,Y))$ is exactly the Cartesian product $\CC(A,B) \times \D(X,Y)$, with the usual  component-wise composition. 

It remains to consider summation on homsets. The projections $\pi_1 : \CC \times \D \rightarrow \CC$, and $\pi_2 : \CC\times \D \rightarrow \D$ are defined exactly as in the usual product of categories. For non-empty $I$, a family $\{ f_i \in (\CC\times \D) ((A,X),(B,Y)) \}_{i\in I}$ is summable exactly when 
\[ \{ \pi_1(f_i) \in \CC(A,B) \}_{i\in I} \ \mbox{ and } \ \{ \pi_2 (f_i ) \in \D(X,Y) \}_{i\in I} \]
are summable, in which case 
\[ \sum_{i\in I} f_i \ = \ \left( \sum_{i\in I} \pi_1(f_i) \ , \ \sum_{i\in I} \pi_2 (f_i) \right)  \in (\CC \times \D) ((A,X),(B,Y)) \]
When $I$ is empty, we simply take $\sum_{i\in I} f_i = (0_{AB},0_{XY})$.

We now demonstrate that this definition satisfies the required universal property for a categorical product: Given PCM-functors $\Gamma_1\in \CS({\mathcal X},\CC)$ and $\Gamma_2\in \CS({\mathcal X},\D)$, we define  $\langle \Gamma_1 , \Gamma_2 \rangle : {\mathcal X} \rightarrow \CC \times \D$
 by 
 \begin{itemize}
 \item {\em On objects}  $\langle \Gamma_1 , \Gamma_2 \rangle (R) = ( \Gamma_1(R),\Gamma_2(R))$, for all $R\in Ob({\mathcal X})$.
 \item {\em On arrows}  $\langle \Gamma_1 , \Gamma_2 \rangle (f) = ( \Gamma_1(f),\Gamma_2(f))\in (\CC\times \D)((\Gamma_1(R),\Gamma_2(R)),(\Gamma_1(S),\Gamma_2(S))$, for all $f\in {\mathcal X}(R,S)$.
 \end{itemize}
 Functoriality of $\langle \Gamma_1 , \Gamma_2 \rangle$ is immediate. To demonstrate that it is also a PCM-functor, consider a summable family $\{ f_i \in {\mathcal X}(R,S) \}_{i\in I}$. Then 
 \[ \{ \langle \Gamma_1 , \Gamma_2 \rangle (f_i) \}_{i\in I} \ = \ \{ ( \Gamma_1 (f_i) , \Gamma_2 (f_i))  \}_{i\in I} \]
 which is summable by definition of summability in $\CC\times \D$. By the definition of summation in $\CC\times \D$,
 \[ \sum_{i\in I} \langle \Gamma_1 , \Gamma_2 \rangle (f_i) \ = \ \left( \sum_{i\in I} \Gamma_1(f_i) \ , \ \sum_{i\in I} \Gamma_2 (f_i) \right) \]
 and thus $\langle \Gamma_1 , \Gamma_2 \rangle$ is also a PCM-functor.
Finally, by the usual theory of product categories, the following diagram commutes:
\[ \diagram
& {\mathcal X} \dlto_{\Gamma_1} \drto^{\Gamma_2} \dto|{\langle \Gamma_1,\Gamma_2 \rangle} \\
\CC & \CC\times \D \lto^{\pi_1} \rto_{\pi_2} & \D \\
\enddiagram \]
\end{proof}

\section{The categorical Cauchy product}
We are now in a position to introduce a categorical analogue of the monoid semiring construction of Definition \ref{msr}. In honour of the original axiomatisation of such products in the theory of formal power series, we refer to this as the {\em (categorical) Cauchy product\footnote{Some new terminology is certainly needed. Starting from the theory of monoid semirings, we will replace both monoids and semirings with categories. However, we wish to avoid replacing the term `monoid-semiring' by `category-category'.}}. However, we first require the following preliminary definition:

\begin{definition}\label{countablecat}{\em Locally countable categories}\\
We say that a category $\D$ is {\bf locally countable} when, for all $U,V\in Ob({\D})$, the homset $\D(U,V)$ is a countable set. We denote the full subcategory of {\bf Cat}, whose objects are locally countable categories, by $\CAT$. 
\end{definition}

\begin{definition}\label{catcauchy}
Given a PCM-category $\CC \in Ob(\CS )$ and a locally countable category $\D \in Ob(\CAT)$, we define their {\bf Cauchy product} $\CC [\D ]\in Ob(\CS )$ as follows:

\begin{itemize}
\item {\bf Objects} $Ob(\CC [\D ]) = Ob(\mathcal C)\times Ob(\mathcal D)$
\item {\bf Arrows} The homset $\CC [\D ]((X,U),(Y,V))$ consists of all functions 
\[ f : {\mathcal D}(U,V)\rightarrow {\mathcal C}(X,Y) \]
such that $\{ f(a) \in {\mathcal C}(X,Y) \}_{a\in {\mathcal D}(U,V)}$ is a summable family.
\item {\bf Composition} Given $g\in \CC [\D ] ((Y,V),(Z,W))$ and $f\in \CC [\D ]((X,U),(Y,V))$
as functions 
\[ f:\D(U,V)\rightarrow {\CC}(X,Y) \ \mbox{ and } \ g:\D(V,W)\rightarrow {\CC}(Y,Z) \]
then $gf\in \CC [\D ]((X,U),(Z,W))$ is the function from $\D(U,W)$ to ${\CC}(X,Z)$ given by:
\[ gf (c) = \sum_{\{ (b,a) :c=ba \} \subseteq {\D(V,W)\times \D(U,V)} } g(b)f(a)  \]
For clarity, we will often use the shorthand notation 
\[ gf (c) = \sum_{c=ba  } g(b)f(a)  \]
\item {\bf Summation} An indexed family $\{ f_i \in \CC [\D ] ((X,U),(Y,V))\}_{i\in I}$ is summable exactly when 
\[ \{ f_i (h) \in {\CC}(X,Y) \}_{(i,h)\in I \times \D(U,V)} \ \mbox{ is summable in } {\CC} \]
in which case 
\[ \left( \sum_{i\in I}f_i \right)(h) \ \stackrel{def.}{ = } \ \sum_{i\in I}  f_i(h) \ \in \ {\CC}(X,Y) \]
\end{itemize}
\end{definition}

\paragraph*{\bf Terminology} In the above definition of the Cauchy product $\CC [\D ]$, we refer to the PCM-category ${\CC}\in Ob(\CS)$ as the {\bf base category} and the locally countable category $\D\in Ob(\CAT)$ as the {\bf index category}.\\

\noindent
We now prove that the above construction is well-defined: 
\begin{theorem}\label{cauchyproof}The Cauchy product $\CC [\D ]$ defined above is a {\bf PCM}-category.
\end{theorem}
\begin{proof}{\em We first show that $\CC [\D ]$ is a category, and then consider the indexed summation on homsets.}\\
We demonstrate that the composition of $\CC [\D ]$ is well-defined, associative, and has identities:
\begin{enumerate}
\item {\bf Composition is well-defined} Given arrows 
\[ g\in \CC [\D ]((Y,V),(Z,W))\  \ \mbox{ and } \ \ f\in \CC [\D ]((X,U),(Y,V)) \] 
(i.e. functions $f : {\mathcal D}(U,V)\rightarrow {\mathcal C}(X,Y)$ and $g : {\mathcal D}(V,W)\rightarrow {\mathcal C}(Y,Z)$
where $\{ f(a) \in {\mathcal C}(X,Y) \}_{a\in {\mathcal D}(U,V)}$ and $\{ g(b) \in {\CC}(Y,Z) \}_{b\in \D(V,W)}$ are summable), we need to show that $(gf)(c)\in{\CC}(X,Z)$ exists, for all $c\in \D(U,W)$, and 
\[ \left\{ gf (c) = \sum_{c=ba} g(b)f(a) \right\} _{c\in \CC(X,Z)} \] 
is also a summable family.
By definition, $\sum_{a\in \D(U,V)} f(a) \in {\mathcal C}(X,Y)$ exists, as does
$\sum_{b\in \D(V,W)} g(b)\in {\mathcal C}(Y,Z)$.  The strong distributivity property thus implies the summability of the indexed family
\[ P = \{   g(b) f(a)\}_{(b,a)\in \D(V,W) \times  \D(U,V)} \]
together with the identity
\[ \left(\sum_{b\in \D(V,W)}  g(b) \right) \left(\sum_{a\in \D(U,V)} f(a)\right) = \sum \left( P \right)  \] 
Given some arbitrary $c\in \D(U,W)$, consider the  (possibly empty) subfamily of $P_c$ of $P$ given by $\{ g(b)f(a) \}_{ ba=c }$. This is a subfamily of $P$,  and thus is itself a summable family, by the subfamilies property of Proposition \ref{basics}. Therefore $(gf)(c)\in{\CC}(X,Z)$ is well-defined, for all $c\in \D(U,W)$. 

Finally, consider the family $\{ P_c \}_{c\in \D(U,W)}$. Observe that, for distinct $x\neq y \in \D(U,W)$, the intersection of $P_x$ and $P_y$ is empty. Thus, $\{ P_c \}_{c\in \D(U,W)}$ is a partition of the summable family $P$, and by the weak partition-associativity axiom is itself a summable family satisfying 
\[ \sum_{c\in \D(U,W)} P_c = \sum_{(b,a)\in \D(V,W)\times \D(U,V)} g(b)f(a) \]
\item {\bf Associativity of composition} \\
Consider arrows
\begin{itemize}
\item $h\in \CC[\D ]((Z,W),(T,P))$, 
\item $g\in \CC[\D ] ((Y,V),(Z,W))$, 
\item $f\in \CC[\D ]((X,U),(Y,V))$.  
\end{itemize}
By definition, 
\[ (hg)(r) = \sum_{\{(q,p):r=qp \}\subseteq \D (W,P)\times \D(V,W)} h(q)g(p) \] 
and similarly, 
\[ (gf)(c) = \sum_{\{ (b,a) : c=ba \}\subseteq \D(V,W)\times \D(U,V)} g(b)f(a). \]
Therefore, for all $\gamma\in \D(U,P)$, 
\[ \left( h(gf) \right) (\gamma) = \sum_{ \{ (\beta,\alpha) : \gamma=\beta\alpha \}\subseteq \D (W,P)\times \D(U,W) 
}  
h(\beta)\left( gf\right) (\alpha)  \]
which, by definition of the composite $gf\in \CC[\D]((X,U),(Z,W))$ is given by 
\[  \left( h(gf) \right) (\gamma) = \sum_{ \{ (\beta,\alpha) : \gamma=\beta\alpha \}\subseteq \D (W,P)\times \D(U,W) 
}  
h(\beta)\left( \sum _{\{ (b,a) : \alpha = ba \} \subseteq \D (V,W)\times \D (U,V) }  g(b)f(a) \right) .   \]
By distributivity (Proposition \ref{distrib}) we may write this as 
\[  \left( h(gf) \right) (\gamma) = \sum_{ \{ (\beta,\alpha) : \gamma=\beta\alpha \}\subseteq \D (W,P)\times \D(U,W) 
}  
\left( \sum _{\{ (b,a) : \alpha = ba \} \subseteq \D (V,W)\times \D (U,V) }  h(\beta)g(b)f(a) \right) .  \]

\[ \]
Conversely, $((hg)f)\in \CC[\D ]((X,U),(T,P))$ is given by, for all $\nu\in \D(U,P)$,
\[ ((hg)f)(\nu) = \sum_{ \{ (\mu,\lambda ):\nu=\mu\lambda \}\subseteq \D(V,P)\times \D(U,V) } (hg)(\mu)f(\lambda) \]
which, by definition of the composite $hg\in \CC[\D]((Y,V),(T,P))$ is given by 
\[  ((hg)f)(\nu) = \sum_{\{ (\mu,\lambda ):\nu=\mu\lambda \}\subseteq \D(V,P)\times \D(U,V) } \left(
\sum_{ \{ (c,b) : \mu=cb \}\subseteq \D (W,P)\times \D(V,W) } h(c)g(b) 
\right)f(\lambda) .
\]
Again by distributivity (Proposition \ref{distrib}) this may be written as 
\[  ((hg)f)(\nu) = \sum_{\{ (\mu,\lambda ):\nu=\mu\lambda \}\subseteq \D(V,P)\times \D(U,V) } \left(
\sum_{ \{ (c,b) : \mu=cb \}\subseteq \D (W,P)\times \D(V,W) } h(c)g(b) 
f(\lambda) \right)
\]

Now observe that, by definition of arrows of $\CC[\D]$, the families 
\begin{itemize}
\item $\{ h(c) \in \CC(Z,T) \}_{c\in \D(W,P)} $
\item $\{ g(b)\in \CC(Y,Z)  \}_{b\in \D(V,W)}$
\item $\{ f(a) \in \CC((X,Y) \}_{a\in \D(U,V)}$
\end{itemize}
are all summable. Therefore, by the strong distributivity property, the family
\[ \{ h(c)g(b)f(a) \} _{(c,b,a)\in \D(W,P)\times \D(V,W)\times \D(U,V)} \]
is summable. Given arbitrary $d\in \D(U,P)$, let $Q_d$ be the subfamily of the above indexing set given by 
\[ Q_d= \{ (c,b,a) : d=cba \} \subseteq \D(W,P)\times \D(V,W)\times \D(U,V)  . \]
Then by the summable subfamilies property, this is summable. 
By the weak partition-associativity axiom, we may partition 
$\sum_{Q_d} h(c)g(b)f(a)$ in two distinct ways --- by relabelling indices these may be seen to correspond to $((hg)f)(d)$ and $(h(gf))(d)$ respectively. Thus $((hg)f)(d)=(h(gf))(d)$, for all $d\in \D(U,P)$, and thus $(hg)f = h(gf)\in \CC[\D] ((X,U),(T,P))$, as required.
\item {\bf Identity arrows} Recall the existence of zero elements in a PCM, from Proposition \ref{basics}, and the proof that PCM-categories have zero arrows, in Corollary \ref{zeroarrows}.   
At an object $(X,U)\in Ob(\CC [\D ])$, the identity arrow is given by $1_{(X,U)}$ by 
\[ 1_{(X,U)}(r) = \left\{ \begin{array}{lr} 1_X \in {\CC}(X,X) & r=1_U\in \D(U,U) \\ 0_X & \mbox{otherwise.} \end{array}\right. \]
From the definition of composition, and Proposition \ref{basics}, for all $g\in \CC[\D] ((X,U),(Y,V))$ and $f\in \CC[\D]((W,T),(X,U))$,
\[ \left(g1_{(X,U)}\right)(s) = g(s) \ \ \forall s\in \D(U,V) \]
and 
\[ \left(1_{(X,U)}f\right) (r) =f(r) \ \ \forall r\in \D( T,U) \]
Thus $1_{(X,U)}\in \CC[\D]((X,U),(X,U))$ is the identity, as required.
\end{enumerate}
It now remains to show that $\CC[\D]$ is not only a category, but a PCM-category:
\begin{enumerate}
\item {\bf Hom-sets are PCMs} Given objects $(X,U),(Y,V)\in Ob(\CC[\D])$, we now demonstrate the summation given in Definition \ref{catcauchy} above gives a PCM structure to $\CC[D]((X,U),(Y,V))$.
\begin{itemize}
\item {\bf The unary sum axiom} \\
Consider an indexed family 
\[ \{ f_i\in \CC[\D]((X,U),(Y,V)) \}_{i\in \{ i'\} } \ \mbox{ where } \  f_{i'}=f . \]
 We first demonstrate that $\{ f_i(a) \in \CC(X,Y) \}_{(i,a)\in \{ i'\} \times \D(U,V)}$ is summable in $\CC$. As $\{ i'\}$ is a single-element set, $\{ i'\} \times \D(U,V)\cong \D$, and (trivially) $f_i=f$, for all $i\in \{ i'\}$. Therefore, by Proposition \ref{welldefined} the summability of $\{ f_i(a) \in \CC(X,Y) \}_{(i,a)\in \{ i'\} \times \D(U,V)}$ is equivalent to the summability of $\{ f(a) \in \CC(X,Y) \}_{a\in \D(U,V)}$, and this is summable by the definition of arrows in $\CC[\D]$. Thus singleton families are summable in $\CC[\D]((X,U),(Y,V))$. Finally, by definition of the summation of $\CC[\D]$, and the unary sum axiom for the PCM $\left(\CC(X,Y) , \Sigma^{X,Y}\right)$, 
 \[ \left( \sum_{i\in \{ i'\} } f_i \right) (a) = \sum_{i\in \{ i'\} } f_i (a) = f(a) \]
 and therefore $\sum_{i\in \{ i'\} } f_i = f\in \CC[\D]((X,U),(Y,V))$.
\item {\bf Weak partition-associativity}\\ 
Consider a summable family $\{ f_i \in \CC[\D]((X,U),(Y,V)) \}_{i\in I}$, and let $\{ I_j \}_{j\in J}$ be a partition of $I$. By definition of summability in $\CC[\D]$,  the family $\{ f_i(a) \in \CC(X,Y) \}_{(i,a)\in I\times \D(U,V)}$ is summable in $\CC$. Now consider the family $\{ f_{i'}(a) \in \CC(X,Y) \}_{(i',a)\in I_j \times \D(U,V) }$. This is a subfamily of a summable family of $\CC(X,Y)$ and thus, by the summable subfamilies property (Proposition \ref{basics}) is itself a summable family. Therefore, by definition of summability in $\CC[\D]$, the family $\{ f_{i'} \in \CC[\D] ((X,U),(Y,V))\}_{i\in I'}$ is summable. 

Similarly, to show that 
$\sum_{j\in J} \left( \sum_{i'\in I_j} f_{i'} \right)$ is summable in $\CC[\D]((X,U),(Y,V))$, note that $\{ \sum_{i'\in I_j} f_{i'}(a) \in \CC(X,Y) \}_{(j,a)\in J\times \D(U,V)}$ is summable, by the weak partition-associativity axiom for the PCM $\left( \CC(X,Y),\Sigma^{X,Y} \right)$, and (again, by WPA), 
\[ \left( \sum_{j\in J} \left( \sum_{i'\in I_J} f_{i'} \right) \right) (a) = \left( \sum_{i\in I} f_i \right) (a) \ \in \CC(X,Y)  \]
 for all $a\in \D(U,V)$,
and thus 
\[ \left( \sum_{j\in J} \left( \sum_{i'\in I_J} f_{i'} \right) \right) = \left( \sum_{i\in I} f_i \right) \ \in \ \CC[\D]((X,U),(Y,V)) . \]
\end{itemize}
Therefore, the summation on $\CC[\D]((X,U),(Y,V))$ satisfies weak partition-associativity.
\item {\bf The strong distributive law}\\
Consider summable families of $\CC[\D]$
\[ \{ f_i \in \CC[\D] ((X,U),(Y,V)) \}_{i\in I} \ \mbox{ and } \ \{ g_j \in \CC[\D]((Y,V),(Z,W))\}_{j\in J} . \]
Summability of these families is equivalent to the summability of the following families in $\CC$
\[ \{ f_i(a) \in \CC (X,Y) \}_{(i,a)\in I\times \D(U,V)} \ \mbox{ and } \ \{ g_j(b) \in \CC(Y,Z)\}_{(j,b)\in J\times \D(V,W)} . \]
By the strong distributivity law for $\CC$, the following family is therefore summable:
\[ \{ g_j(b)f_i(a) \in \CC(X,Z) \}_{(j,b,i,a)\in J\times \D(V,W) \times I \times \D(U,V)} . \]
For all $c\in \D(U,W)$, consider the (possibly empty) subset 
\[ P_c= \{ (j,b,i,a) : c=ba \} \subseteq J\times \D(V,W)\times I \times \D(U,V) . \]
Note that $P_c\cap P_{c'}=\emptyset$, for all $c\neq c'$, and 
\[ \bigcup_{c\in \D(U,W)} P_c = J\times \D(V,W)\times I \times \D(U,V) \]
giving a $\D(U,W)$-indexed partition of the summable family 
\[ \{ g_j(b)f_i(a) \in \CC(X,Z) \}_{(j,b,i,a)\in J\times \D(V,W) \times I \times \D(U,V)}  . \]
Thus, by the weak partition-associativity property of $\left( \CC(X,Z),\Sigma^{X,Y} \right) $, the family 
\[ \{ g_j(b)f_i(a) \in \CC(X,Z) : ba=c \}_{(j,c,i)\in J\times \D(U,W) \times I } \]
is summable, demonstrating that $\{ g_jf_i \in \CC[\D]((X,U),(Y,V)) \}_{(j,i)\in J\times I} $ is summable in $\CC[\D]$, as required.

For all $c\in \D(U,W)$, the identity 
\[ \left( \sum_{j\in J} g_j \right) \left( \sum_{i\in I} f_i \right)(c) = \left( \sum_{(j,i)\in J\times I} g_jf_i \right) (c) \ \in \CC(X,Z) \]  
is then immediate from the existence of both sides of this equation, and the strong distributivity law for $\CC$, and so
\[ \left( \sum_{j\in J} g_j \right) \left( \sum_{i\in I} f_i \right) = \left( \sum_{(j,i)\in J\times I} g_jf_i \right)  \ \in \CC[\D] ((X,U),(Z,W)) \]
as required.
\end{enumerate}
\end{proof}

\subsection{The Cauchy product as a bifunctor}

\begin{theorem}
The Cauchy product of Definition \ref{catcauchy} defines a bifunctor 
\[ (\underline{\ \ })[\underline{\ \ }] : \CS\times \CAT
\rightarrow \CS \]
That is
\begin{enumerate}
\item Given $\D\in Ob(\CAT)$, then $(\underline{\ \ })[\D] :\CS\rightarrow \CS$ is a functor.
\item  Given $\CC\in Ob({\CS})$, then $\CC [\underline{ \ \ }] : \CAT\rightarrow \CS$ is a functor.
\end{enumerate}
\end{theorem}
\begin{proof}
\begin{enumerate}
\item
We first demonstrate that for arbitrary $\D\in \CAT$, the map $(\underline{\ \ }) [\D ] : \CS \rightarrow \CS$ defines a functor. 
\begin{itemize}
\item {\bf on Objects} Given a PCM-category $\CC \in Ob(\CS)$, then $\CC [\D ]\in Ob(\CS)$  is as defined in Definition \ref{catcauchy}.
\item {\bf on Arrows} Given $\Gamma \in \CS (\CC , \E)$, we define the functor $\left( \Gamma [\D] \right) \in \CS(\CC [\D ],\E [\D ])$ as follows:
\begin{itemize}
\item {\em on Objects} For all $(X,U)\in Ob(\CC [\D ])$, we define $\left( \Gamma [\D] \right) (X,U)=(\Gamma(X),U)$.
\item {\em on Arrows} Given an arbitrary arrow $f\in \CC [\D ]((X,U),(Y,V))$, 
we define $\left( \Gamma [\D] \right)(f) \in \E [\D] ((\Gamma(X),U),(\Gamma(Y),V))$ by, for all $r\in \D(U,V)$, 
\[ \left( \Gamma [\D] \right)(f)(r) = \Gamma (f(r)) \in \E(\Gamma(X),\Gamma(Y)) \]
It is immediate that this is well-defined as an arrow in $\E [\D ]((\Gamma(X),U),(\Gamma(Y),V))$ since, as  $\Gamma$ is a $\PCM$-functor (i.e. an arrow of $\CS(\CC,\E)$) 
\[ \sum_{r\in \D(U,V)} f(r) \ \mbox{ exists } \ \ \Rightarrow \ \ \sum_{r\in \D(U,V)} \Gamma (f(r)) \ \mbox{ exists.} \]
$ $ \\

To prove compositionality, consider 
\[ f\in \CC[\D] ((X,U),(Y,V)) \ \mbox{ and } \ g\in \CC[\D]((Y,V),(Z,W)) \]
By definition of composition, $gf(c) = \sum_{c=ba}g(b)f(a)$, for all $c\in \D(U,W)$. However, by definition of the functor $\left( \Gamma [\D] \right)$, 
\[ \left( \left( \Gamma [\D] \right)(g)\left( \Gamma [\D] \right)(f) \right) (c) = \sum_{c=ba} \Gamma(g(b))\Gamma(f(a)) \]
By functoriality of $\Gamma$, 
\[ \left( \left( \Gamma [\D] \right)(g)\left( \Gamma [\D] \right)(f) \right) (c) = \sum_{c=ba} \Gamma(g(b)f(a)) \]
and as $\Gamma$ is a PCM-functor, 
\[ \left( \left( \Gamma [\D] \right)(g)\left( \Gamma [\D] \right)(f) \right) (c) = \Gamma \left( \sum_{c=ba} g(b)f(a) \right) = \left( \Gamma [\D]\right) (gf)(c) \]
\end{itemize}
Finally, given another functor $\Delta \in \CS(\E,\F)$, then 
\begin{itemize}
\item On objects: 
\[ \left( \Delta [\D]\right) \left( \Gamma [\D]\right) (X,U) = (\Delta \Gamma (X),U) = \left( (\Delta \Gamma) [\D]\right) (X,U) \]
\item On arrows: given $f\in \CC [\D ]((X,U),(Y,V))$, then 
\[ \left( \Delta [\D]\right) \left( \Gamma [\D]\right) (f)(r) = \Delta (\Gamma (f))(r) = (\Delta \Gamma (f))(r) = 
\left( (\Delta \Gamma) [\D]\right) (f)(r) \]
\end{itemize}

\end{itemize}
\item We now demonstrate that for arbitrary $\CC\in \CS$, the map $\CC[\underline{ \ \ }] : \CAT\rightarrow \CS$ is also functorial.
\begin{itemize}
\item {\bf on Objects} Given arbitrary $\D\in Ob(\CAT)$, then $\CC [\D ]$ is given in Definition \ref{catcauchy}.
\item {\bf on Arrows} Given a functor $\Lambda\in \CAT(\D,\HH)$, we define $\CC[\Lambda ]\in \CS(\CC [\D ],\CC[\HH])$ by:
\begin{itemize}
\item {\em on Objects} $\CC[\Lambda] (X,U) = (X,\Lambda(U))$.
\item {\em on Arrows}, given $f\in \CC [\D ]((X,U),(Y,V))$, we define 
\[ \left( \CC[\Lambda]\right)(f) \in \CC[\HH] ((X,\Lambda(U)),(Y,\Lambda(V))) \]
by, for all $x\in \HH(\Lambda (U),\Lambda (V))$,
\[ \left( \CC[\Lambda ] \right) (f) (x) = \sum_{x=\Lambda (a) } f(a) \ \in \CC(X,Y) \]
 The above sum is well-defined, since $\{ f(a) \}_{a\in \D(U,V)}$ is a summable family. Also, note that $\{ x=\Gamma (a) \}_{x\in \HH (\Gamma (U),\Gamma( V))}$ is a partition of $\D(U,V)$, and thus, by the weak partition-associativity axiom, $\{ \left( \CC[\Lambda ] \right) (f) (x) \}_{x\in \HH(\Lambda (U),\Lambda (V))}$ is summable, and so $\left( \CC[\Lambda ] \right) (f) \ \in \ \CC[\HH]((X,\Lambda (U)),(Y,\Lambda (V)))$ is well-defined.\\
 
 To prove compositionality, consider 
 \[ f\in \CC[\D] ((X,U),(Y,V)) \ \mbox{ and } \ g\in \CC[\D]((Y,V),(Z,W)) \]
By definition of composition, $gf(c) = \sum_{c=ba}g(b)f(a)$, for all $c\in \D(U,W)$, and hence, for all $z\in \HH(\Lambda (U),\Lambda (W))$, 
\[ \left( \CC(\Lambda)(gf)\right) (z) = \sum_{z=\Lambda (c)} (gf)(c) \]
Now note that, for all $y\in \HH(\Lambda (V),\Lambda (W))$ and $x\in \HH (\Lambda (U),\Lambda (V))$,
\[ \left( \CC(\Lambda)(g)\right) (y) = \sum_{y=\Lambda (b)} g(b) \ \ \mbox { and } \ \ 
\left( \CC(\Lambda)(f)\right) (x) = \sum_{x=\Lambda (a)} f(c) \]
and thus, by the strong distributive law for PCM-categories, and the functoriality of $\Lambda$,  
\[ 
\left( \CC(\Lambda)(g)\right) \left( \CC(\Lambda)(f)\right) (z) = \sum_{z=\Lambda (c)} (gf)(c) = \left( \CC(\Lambda)(gf)\right) (z) \ \in \CC(X,Z) \] 
\end{itemize}
\end{itemize}
Finally, given another functor $\Omega \in \CAT(\HH,\K)$, then 
\begin{itemize}
\item On objects:
\[ \left( \CC [\Omega]\right) \left( \CC[\Lambda ] \right) (X,U) = (X,\Omega \Lambda (U)) = \left( \CC [\Omega \Lambda ] \right) (X,U) \]
\item On arrows: given $f\in \CC [\D ]((X,U),(Y,V))$, then for all $p\in \K (\Omega \Lambda (U), \Omega \Lambda (V))$, 
\[ \left( \CC[\Omega ]\right)\left( \CC[\Lambda ]\right) (f) (p) = \sum_{p=\Omega(x) \ , \ x=\Lambda (a)} f(a) = 
\sum_{p=\Omega\Lambda (a)} f(a) = \left( \CC[\Omega \Lambda ]\right) (f) (p) \]
\end{itemize}
\end{enumerate}
\end{proof}

\paragraph{\bf Is the Cauchy product a monoidal tensor?}
Since the Cauchy product is a bifunctor $\CS\times\CAT\rightarrow \CS$, it is natural to wonder whether, when restricted to locally countable PCM-categories, it is in fact a monoidal tensor. It is also easy to show that this is not the case: consider three locally countable PCM-categories, $\CC,\D,\E \in Ob(\CS)$, and denote their (object-indexed families of) summations by 
$\Sigma^{\CC(\underline{\ } , \underline{\ })}$, $\Sigma ^{\D(\underline{\ } , \underline{\ })} $, $\Sigma^{\E(\underline{\ } , \underline{\ })}$ respectively. Then it is immediate that the structure of $\CC [ \D[\E ]]$ depends on the family of summations $\Sigma^ {\D(\underline{\ } , \underline{\ })} $ on the homsets of $\D$, 
whereas the structure of $(\CC [\D ])[\E ]$ is independent of $\Sigma ^{\D(\underline{\ } , \underline{\ })}$. Therefore, in general, $(\CC [\D ])[\E ]$ cannot be equal to $\CC [ \D[\E ]]$, even up to a canonical isomorphism.

Rather, as we now demonstrate, there exist embeddings of $\CC$ into $\CC[\D]$ indexed by objects of $\D$, together with embeddings of $\D$ into $\CC[\D]$ indexed by objects of $\CC$, and an embedding of the product $\CC\times \D$ into $\CC [\D ]$. The embeddings of $\CC$ into $\CC[\D]$ also have a common left-inverse, giving an indexed family of retractions.

\section{Embedding the base category into a Cauchy product}
We now give an embedding of the base category $\CC$ into the Cauchy product $\CC [\D ]$, and show that $\CC$ is a retract of $\CC [\D ]$.

We first exhibit a forgetful functor from $\CC [\D ]$ to $\CC$:
\begin{definition} Given $\D\in Ob(\CAT)$, and ${\CC}\in Ob(\CS)$, we define $\sigma_{\CC,\D}:\CC [\D ]\rightarrow {\CC}$ by 
\begin{itemize}
\item {\bf (on objects)} $\sigma_{\CC,\D}(X,U)=X$, for all $(X,U)\in Ob(\CC [\D ])$.
\item {\bf (on arrows)} Given $h\in {\CC} [D] ((X,U),(Y,V))$,  then 
\[ \sigma_{\CC,\D}(h) = \sum_{a\in \D(U,V)} h(a) \ \in \ {\CC}(X,Y) . \]
\end{itemize}
\end{definition}

\begin{proposition} $\sigma_{\CC,\D} : \CC[\D] \rightarrow \CC$, as given above, is a PCM-functor.
\end{proposition}
\begin{proof}
First note that, by definition of arrows in $\CC[\D]$,  the family 
$\{ h(a) \}_{a\in \D(U,V)}$ is summable for all $h\in {\CC} [D] ((X,U),(Y,V))$, and so $\sigma_{\CC,\D}(h) = \sum_{a\in \D(U,V)} h(a) \ \in \ {\CC}(X,Y)$ is well-defined.

To prove functoriality, consider $k\in \CC[\D]((Y,V),(Z,W))$. Then 
\[ \sigma(k)\sigma(h) = \left( \sum_{b\in \D(V,W)} k(b) \right)\left( \sum_{a\in \D(U,V)} h(a) \right) . \]
By the strong distributivity property, 
\[ \sigma(k)\sigma(h) =  \sum_{(b,a)\in \D(V,W)\times \D(U,V)} k(b)h(a)  . \]
Conversely, $kh\in \CC[\D]((X,U)(Z,W))$ is defined by, for all $c\in \D(U,W)$,
\[ (kh)(c) = \sum_{c=ba}k(b)h(a) .  \]
Now note that $\{ (kh)(c) \}_{c\in \D(U,W)}$ is a summable family, and  by weak partition-associativity,
\[ \sigma(kh) = \sum_{c=ba}k(b)h(a) =  \sum_{(b,a)\in \D(V,W)\times \D(U,V)} k(b)h(a)  = \sigma(k)\sigma(h) .  \]
Thus $\sigma : \CC[\D]\rightarrow \CC$ preserves composition. The proof that it also preserves identities follows from the formula for identities in PCM-categories given in Theorem \ref{cauchyproof}, 
\[ 1_{(X,U)}(r) = \left\{ \begin{array}{lr} 1_X \in {\CC}(X,X) & r=1_U\in \D(U,U) \\ 0_X & \mbox{otherwise.} \end{array}\right. \] 
It is immediate that $\sigma \left( 1_{(X,U)} \right) =1_X\in \CC(X,X)$.

Now consider a summable family $\{f_i \in \CC[\D]((X,U),(Y,V)) \}_{i\in I}$. By definition of summation in $\CC[\D]$,  the family $\{ f_i (a) \in {\CC}(X,Y) \}_{(i,a)\in I \times \D(U,V)}$  is summable in $\CC$, and 
again by definition 
\[ \left( \sum_{i\in I}f_i \right)(a) \  =  \ \sum_{i\in I}  f_i(a) \ \in \ {\CC}(X,Y) . \]
Thus, by weak partition-associativity, and Proposition \ref{welldefined},
\[  \sum_{i\in I} \sigma(f_i) = \sum_{i\in I} \left( \sum_{a\in \D(U,V)} f(a) \right)  = \sum_{a\in \D(U,V)} \left( \sum_{i\in I} f_i(a) \right) = \sigma \left( \sum_{i\in I} f_i \right) . \]
Therefore, $\sigma : \CC[\D]\rightarrow \CC$ is a PCM-functor.
\end{proof}

We now exhibit a family of embeddings of $\CC$ into $\CC [\D ]$, indexed by objects of $\D$:
\begin{definition}\label{baseinjection}
Let $\D$ be an arbitrary category, and let ${\CC}$ be a PCM-category. 
For all $U\in Ob(\D)$, we define $\eta_{\CC,U} :{\CC}\rightarrow \CC [\D ]$ by 
\begin{itemize}
\item {\bf (on objects)} $\eta_{\CC,U} (X)=(X,U)$, for all $X\in Ob({\CC})$.
\item {\bf (on arrows)} Given $h\in {\CC} (X,Y)$, then $\eta_{\CC, U}(h)\in \CC [\D ] ((X,U),(Y,U))$ is the function $\eta_{\CC,U} (h) : \D(U,U)\rightarrow {\CC}(X,Y)$ given by 
\[ \left( \eta_{\CC,U}(h)\right)  (a) = \left\{ \begin{array}{lr} h &  a=1_U \\ 0_{XY} & a\neq 1_U \end{array}\right. \]
\end{itemize}
\end{definition}
We prove that these maps are injective PCM-functors.
\begin{proposition} For all $U\in Ob(\D)$, the map 
$\eta_{\CC,U} :{\CC}\rightarrow \CC [\D ]$ defined above is an injective PCM-functor.
\end{proposition}
\begin{proof}
Given $h\in \CC(X,Y)$ then $\eta_U(h)$ is trivially well-defined as an arrow of $\CC[\D]((X,U),(Y,U))$, since $\sum_{a\in \D(U,U)} h(a)=h$, by Proposition \ref{basics}. 
Now consider $k\in \CC(Y,Z)$. By definition of composition in $\CC[\D]$, 
\[ \left( \eta_{\CC,U} (k) \right) \left( \eta_{\CC,U} (h)\right) (c) = \sum_{c=ba} \left( \eta_{\CC,U} (k)\right)(b)
\left( \eta_{\CC,U} (h)\right)(a) \]
However, 
\[ \left( \eta_{\CC,U} (k)\right)(b)
\left( \eta_{\CC,U} (h)\right)(a)  =  \left\{ \begin{array}{lr} kh &  b=a=1_U \\ 0_{XY} & \mbox{otherwise}\end{array}\right. \]
and therefore
\[ \left( \eta_{\CC,U} (k) \right) \left( \eta_{\CC,U} (h)\right) (c) = 
\left\{ \begin{array}{lr} kh &  c=1_U \\ 0_{XY} & c\neq 1_U \end{array}\right. \]
giving $\eta_{\CC,U} (k) \eta_{\CC,U}(h) = \eta_{\CC,U}(kh)$ as required. 
It is immediate from the definition that $\eta_{\CC,U}(1_X) = 1_{(X,U)}\in \CC[\D]((X,U),(X,U))$, for all $X\in Ob(\CC)$.

For the summation, consider a summable family $\{ f_i\in \CC(X,Y)\}_{i\in I}$. By definition of summability in $\CC[\D]$, the family $\{ \eta_{\CC,U}(f_i) \in \CC[\D] \}_{i\in I}$ is also summable, and 
\[ \sum_{i\in I} \eta_{\CC,U}(f_i)  \ = \ \eta_{\CC,U} \left( \sum_{i\in I} f_i \right) \]

The injectivity of $\eta_{\CC,U} :{\CC}\rightarrow {\CC} [U]$ on objects is immediate. To demonstrate injectivity on arrows, consider 
$f,f'\in \CC(X,Y)$ satisfying $\eta_{\CC,U}(f) = \eta_{\CC,U}(f')$. Then, for all $a\in \D(U,U)$,  
\[ \eta_{\CC,U}(f)(a) = \eta_{\CC,U}(f')(a) \]
Taking $a=1_U$ gives $f=f'$, as required.
\end{proof}
 We therefore have a family of injective PCM-functors from ${\CC}$ to $\CC [\D ]$  indexed by the objects of $\D$.

\begin{proposition}
Let $\D$ be an arbitrary category, and let ${\CC}$ be a PCM-category. Then there exists a family of retractions from $\CC $ to $\CC  [\D ]$, indexed by objects of $\D$. 
\end{proposition}
\begin{proof}
For arbitrary $U\in Ob(\D)$, we demonstrate that $\sigma_{\CC,\D} \eta_{\CC,U}  = Id_{\CC}$: 
\begin{itemize}
\item On objects: 
\[ \sigma_{\CC,\D} \eta_{\CC,U} (X) = \sigma_{\CC,\D}(X,U)=X \]
\item On arrows: given $f\in \CC(X,Y)$, then 
\[ \sigma_{\CC,\D}(\eta_{\CC,U}(f)) = \sum_{a\in \D(U,U)}(\eta_{\CC,U}(f)) \ \ \mbox{ where } \eta_{\CC,U}(f)(a) = \left\{ \begin{array}{lr} f &  a=1_U \\ 0_{XY} & \mbox{otherwise}\end{array}\right. \]
and so by proposition \ref{basics}, $\sigma_{\CC,\D}(\eta_{\CC,U}(f))=f\in \CC(X,Y)$.
\end{itemize}
Thus  $\sigma_{\CC,\D}: \CC [\D ]\rightarrow \CC$ is left-inverse to \underline{all} $\eta_{\CC,U}  : \CC\rightarrow \CC [\D ]$, and so $\CC$ is a retract of $\CC [\D ]$, with retractions indexed by $U\in Ob(\D)$. 
\end{proof}

\section{Embedding the index category into a Cauchy product} 
Similar to the way in which there exists an (object-indexed) family of embeddings of the base category into a Cauchy product, we now exhibit a family of embeddings of the index category into a Cauchy product, indexed by objects of the base category:

\begin{definition}\label{indexinjection}
Let $\D$ be an arbitrary category, and let ${\CC}$ be a PCM-category. For all $X\in Ob({\CC})$ we define the functor $\gamma_{X,\D} : \D\rightarrow \CC [\D ]$ by 
\begin{itemize}
\item {\bf (On objects)} $\gamma_{X,\D}(U)=(X,U)$, for all $U\in Ob(\D)$, 
\item {\bf (On arrows)} Given $h\in \D(U,V)$, then $\gamma_{X,\D}(h)\in\CC [\D ]((X,U),(X,V))$ is defined by 
\[ \gamma_{X,\D}(h)(a) = \left\{ \begin{array}{lr} 1_X  & a=h \\ 0_X & \mbox{otherwise.} \end{array}\right. \]
\end{itemize}
\end{definition}

\begin{proposition}
 $\gamma_{X,\D} : \D\rightarrow \CC [\D ]$,as defined above, is an injective functor for all $X\in Ob(\CC)$.
\end{proposition}
\begin{proof}
By Proposition \ref{basics}, it is immediate that, for all $h\in \D(U,V)$, the family $\{ \gamma_{X,\D}(h)(a) \in \CC(X,X) \}_{a\in \D(U,V)}$ is summable, and hence $\gamma_{X,\D}(h)$  is an arrow of $\CC[\D]((X,U),(X,V))$. To demonstrate functoriality, consider $k\in \D(U,V)$. Then 
\[ \left( \gamma_{X,\D}(k) \gamma_{X,\D}(h) \right) (c) = \sum_{c=ba} \gamma_{X,\D}(k)(b) \gamma_{X,\D}(h)(a) \]
However, 
\[ \gamma_X(k)(b) \left\{ \begin{array}{lr} 1_X & b=k \\ 0_{XX} & \mbox{otherwise} \end{array}\right. \] 
and  
\[ \gamma_X(h)(a) \left\{ \begin{array}{lr} 1_X & a=h \\ 0_{XX} & \mbox{otherwise} \end{array}\right.
\]
Therefore, 
\[ \left( \gamma_X(k) \gamma_X(h) \right) (c) = \left\{ \begin{array}{lr} 1_X & c=kh \\ 0_{XX} & \mbox{otherwise} \end{array} \right. \]
and thus $\gamma_{X,\D}(k)\gamma(X,\D)(h) = \gamma_{X,\D}(kh)$. The proof that $\gamma_{X,\D}$ also preserves identities is trivial.

To demonstrate injectivity, consider $h,h'\in \D(U,V)$. Then 
\[ \gamma_X(h)(a) \left\{ \begin{array}{lr} 1_X & a=h \\ 0_{XX} & \mbox{otherwise} \end{array}\right. \]
 and  
\[ \gamma_X(h')(a) \left\{ \begin{array}{lr} 1_X & a=h' \\ 0_{XX} & \mbox{otherwise} \end{array}\right.
\]
These are identical exactly when $h=h'$.
\end{proof}
We therefore have a family of injective functors from $\D$ to $\CC [\D ]$ indexed by objects of ${\CC}$.

\section{Embedding the product of base and index into the Cauchy product}
As well as the above embeddings of the base and index categories into the Cauchy product, there exist a straightforward  embedding of the product of the base and index categories into the Cauchy product.

\begin{definition}\label{stardef}
Given a PCM-Category $\CC$, and a locally small category $\D$, we define the functor \footnote{Note that this is simply a functor, rather than a PCM-functor, since the product category $\CC\times \D$ is not a PCM-category. However, the construction relies on the assumption that $\CC$ is indeed a PCM-category.}
\[ (\underline{\ \ } \star \underline{\ \ }): \CC\times \D \rightarrow \CC[\D] \]  
as follows:
\begin{itemize}
\item {\bf (Objects)} Given $X\in Ob(\CC)$ and $U\in Ob(\D)$, then $X\star U = (X,U)\in Ob(\CC[\D])$
\item {\bf (Arrows)} Given $f\in \CC(X,Y)$ and $g\in \D(U,V)$, then $f\star g \in \CC[\D]((X,U),(Y,V))$ is the function 
\[ (f\star g)(h) \ = \ \left\{\begin{array}{lr} f & g=h \\  0_{X,Y} & \mbox{otherwise.} \end{array}\right. \]
It is immediate that the family $\left\{ (f\star g)(h) \right\}_{h\in \D(U,V)} $ is summable, and hence this is indeed an arrow of $\CC[\D]$. 
\end{itemize}
\end{definition}

\begin{lemma} The operation $(\underline{\ \ } \star \underline{\ \ }): \CC\times \D \rightarrow \CC[\D]$ is indeed a functor.\end{lemma}
\begin{proof}$ $ \\
{\bf Compositionality} Consider arrows 
\[ f\in \CC(X,Y) \ , \ f'\in \CC (Y,Z) \ \ , \ \ g\in \D(U,V) \ , \ g'\in \D(V,W) \]
By definition, $f'f \star g'g \in \CC[\D]((X,U),(Z,W))$ is given by the function 
\[ (f'f\star g'g)(k) \ = \ \left\{\begin{array}{lr} f'f & k=g'g \\  0_{X,Z} & \mbox{otherwise.} \end{array}\right. \]
Similarly, the arrows $f\star g\in \CC[\D]((X,U),(Y,V))$ and $ f'\star g' \in \CC[\D]((Y,V),(Z,W))$ are, by definition, the functions 
\[ (f\star g)(h) \ = \ \left\{\begin{array}{lr} f & h=g \\  0_{X,Y} & \mbox{otherwise.} \end{array}\right. \]
and 
\[ (f'\star g')(j) \ = \ \left\{\begin{array}{lr} f' & j=g' \\  0_{Y,Z} & \mbox{otherwise.} \end{array}\right. \]
Then, by definition of composition in $\CC[\D]$
\[ \left( (f'\star g')(f\star g)\right) (k) \ = \ \sum_{k=jh} (f'\star g')(j)(f\star g)(h) \]
From the definition of $ f'\star g'$ and $f\star g$, we see that 
\[ \sum_{k=jh} (f'\star g')(j)(f\star g)(h) \ = \ \left\{\begin{array}{cc} f'f & k=g'g \\ 0_{X,Z} & \mbox{otherwise.} \end{array}\right. \]
Hence, $\left( (f'\star g')(f\star g)\right) (k) = (f'f\star g'g)(k)$, as required.
\\

\noindent{\bf Identities} By definition, $1_X \star 1_U\in \CC[\D]((X,U),(X,U))$ is the function 
\[ (1_X\star 1_U)(h) \ = \ \left\{\begin{array}{cc} 1_X & h=1_U \\ 0_{XX} & \mbox{otherwise.} \end{array} \right. \]
which, from Theorem \ref{cauchyproof}, is precisely $1_{(X,U)}\in \CC[\D]((X,U),(X,U))$.
\end{proof}

\begin{corol}
The functor $(\underline{ \ \ } \star\underline{\ \ }) : \CC \times \D \rightarrow \CC[\D]$ is an embedding of $\CC\times \D$ into $\CC[\D]$.
\end{corol}
\begin{proof}
By definition $X\star U \stackrel{def}{=} (X,U)$, and hence $(\underline{ \ \ } \star\underline{\ \ })$ is bijective on objects. 
To see injectivity on arrows, recall that $(f\star g),(h\star k)\in \CC[\D]((X,U),(Y,V))$ are functions from $\D(U,V)$ to $\CC(X,Y)$ 
given in Definition \ref{stardef}. From this definition, these are identical exactly when $f=h$ and $g=k$, and hence are equal in $\CC\times \D$  -- i.e.
\[ f\star g \ =\  h\star k \ \ \Leftrightarrow \ \  f\times g \ = \ h \times k \]
\end{proof}

\section{Universal properties, and the Cauchy product}
The functors $\eta_{\CC,U} :\CC\rightarrow \CC[\D]$ and $\gamma_{X,\D}:\D\rightarrow \CC[\D]$ are clearly (object-indexed)  categorical analogues of the usual {\em augmentation} and {\em inclusion} maps used in the demonstration of the universal property of the monoid semiring construction (see \cite{MS} for a good exposition, albeit in the special case of monoid rings). It is natural to wonder whether an analogous property holds for the categorical Cauchy product. In the following sections, we give an exposition of the usual universal property of monoid semirings, and demonstrate that a straightforward generalisation to the Cauchy product is not possible, except in the trivial one-object case. The computational significance of this is discussed, and we  consider the additional structure that would be required in order to have a suitable universal property in the full multi-object case.

\subsection{The universal property of monoid semirings}
The universal property of monoid semirings is a canonical example of a universal property (see, for example, \cite{MS}). 

\begin{theorem}\label{semiuniversal}
Let $(M,\cdot)$ be a monoid, and $(A,\times,+,1_A,0_A)$ and $(B,\times_B,1_B,0_B)$ be unital semirings. Further, let $f:A\rightarrow B$ be a unital semiring homomorphism, and let $g:(M,\cdot ) \rightarrow (B,\times_B)$ be a monoid homomorphism. Finally, let 
$\eta:A\rightarrow A[M]$ and $\gamma :M\rightarrow A[M]$ be the usual augmentation and inclusion maps. Then  there exists a unique unital semiring homomorphism $h:A[M]\rightarrow B$ such that the diagram of Figure \ref{semiringuniversal} commutes.

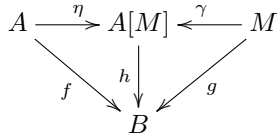
\begin{figure}[b]
\caption{The universal property for monoid semirings}
\label{semiringuniversal}
\[ \xymatrix{
A \ar[r]^{\eta}\ar[dr]_f & A[M] \ar[d]_{h} & M \ar[l]_{\gamma} \ar[dl]^g \\
& B & \\
}
\]
\end{figure}

\end{theorem}
\begin{proof}
The semiring homomorphism $h:A[M]\rightarrow B$ is defined as follows:

Given $\alpha: M\rightarrow A$, an element of $A[M]$, then 
\[ h(\alpha) \ = \ \sum_{m\in M} f(\alpha(m)) g(m) \]
The proof that this is a unique unital semiring homomorphism that makes the above diagram commute is then straightforward, and may be found in many algebra texts (e.g. \cite{MS}).
\end{proof}

\begin{remark}{Interpretation}\label{universalinterpretation}
The usual interpretation of the above universal property is that the unique unital semiring homomorphism making the diagram of Figure \ref{semiringuniversal} commute describes the computational process of {\em instantiating a free variable}. In order to provide motivation for this interpretation, we describe a simple example.

Consider the monoid ${\Bbb Z}_p=\{ 0,1,\ldots ,p-1\}$ for some prime $p$, along with the semiring of natural numbers $\Bbb N$. The members of the monoid semiring $\Bbb N[\Bbb Z_p]$  are often written as `polynomials' in some formal variable $z$, so the function  $f:\Bbb Z_p\rightarrow \Bbb N$ would be written as 
\[ f(0)z^0 + f(1)z^1+f(2)z^2+\ldots + f(p-1)z^{p-1} \]
with the understanding that multiplication of this formal variable is defined by $z^az^b=z^{a+b\ (mod\ p)}$.
Using this polynomial formalism, the augmentation $\eta : \Bbb N\rightarrow \Bbb N[\Bbb Z_p]$  and inclusion $\gamma : \Bbb Z_p\rightarrow \Bbb N[\Bbb Z_p] $  are given by 
\[ \eta (n)= nz^0 \ \ , \ \ \gamma(r) = z^r \]

Let us now consider the usual semiring homomorphism $\iota : \Bbb N\rightarrow \Bbb C$ given by the canonical inclusion 
together with the monoid homomorphism $\chi_s: \Bbb Z_p \rightarrow \Bbb C$ defined by $\chi_s(a)=e^{2\pi i \frac{as}{p}}$ for some fixed $s\neq 0  \in \Bbb N$.

The universal property of the monoid semiring construction tells us that there is a unique induced universal map $sub: \Bbb N[\Bbb Z_p] \rightarrow \Bbb C$
that makes the diagram of Figure \ref{NZPuniversal} commute.
\begin{figure}[t]
\caption{A canonical example of the universal arrow}
\label{NZPuniversal}
\[ \xymatrix{
\Bbb N \ar[r]^{\eta}\ar[dr]_\iota & \Bbb N[\Bbb Z_p] \ar[d]|{sub} & \Bbb Z_p \ar[l]_{\gamma} \ar[dl]^{\chi_s} \\
& \Bbb C & \\
}
\]
\end{figure}
From the prescription given in the proof of Theorem \ref{semiuniversal}, it is immediate that the action of the universal arrow is given by 
\[ \xymatrix{
f(0)z^0\ +\ f(1)z^1\ +\ f(2)z^2\ +\ \ldots \ + \ f(p-1)z^{p-1} \ar@{|->}[d]|{sub} \\
f(0)\ +\ f(1)e^{2\pi i \frac{s}{p}}\ +\ f(2)e^{2\pi i \frac{2s}{p}}\ +\ \ldots \ + \ f(p-1)e^{2\pi i \frac{(p-1)s}{p}}
}
\]
(Note that we elide the inclusion homomorphism $\iota : \Bbb N\rightarrow \Bbb C$, for clarity).
Thus, the induced universal map simply interprets as substituting a concrete value for the formal variable $z$.
\end{remark}

An immediate question is whether such a property also exists for the categorical Cauchy product? That is, given a PCM-functor $\Gamma \in \CS(\CC , \E)$ together with $\D\in Ob(\CAT)$ and a functor $\Delta \in \cat(\D,\E)$, does there exist an object-indexed family of functors $\Upsilon_{X,U}  \in \CS(\CC[\D] , \E)$ making the diagram of Figure \ref{epicfail} commute, for arbitrary choice of $X\in Ob(\CC)$ and $U\in Ob(\D)$?

\begin{figure}[b]
\caption{Can such a universal family of functors exist?}
\label{epicfail}
\[ \xymatrix{
\CC \ar[r]^{\eta_{\CC,U}}\ar[dr]_\Gamma & \CC [\D] \ar[d]|{\Upsilon_{X,U}} & \D \ar[l]_{\gamma_{X,\D}} \ar[dl]^\Delta \\
& \E & \\
}
\]
\end{figure}

However, it is straightforward that, simply because the categorical version is the multiple-object setting, such a universal property can only ever hold in a very restricted setting, as the following result demonstrates:

\begin{proposition}\label{notypeduniversals}
Let us assume the existence of an object-indexed family of functors $\Upsilon_{X,U} : \CC[\D]\rightarrow \E$ making the diagram of Figure \ref{epicfail} commute, for all $X\in Ob(\CC)$ and $U\in Ob(\D)$. Then $\Gamma (X)=\Delta(U)$, for all $X\in Ob(\CC)$ and $U\in Ob(\D)$. 
\end{proposition}
\begin{proof}
Let us fix some arbitrary $X\in Ob(\CC)$ and $Y\in Ob(\D)$, so that the following diagram commutes:
\[ \xymatrix{
\CC \ar[r]^{\eta_{\CC,U}}\ar[dr]_\Gamma & \CC [\D] \ar[d]|{\Upsilon_{X,U}} & \D \ar[l]_{\gamma_{X,\D}} \ar[dl]^\Delta \\
& \E & \\
}
\]
Then for all $P\in Ob(\CC)$, 
\[ \eta_{\CC,U}(P)\ =\   (P,U)\in Ob(\CC[\D]) \]
by definition of $\eta_{\CC,U}$, and by commutativity of the above diagram 
$\Upsilon_{X,U}\eta_{\CC,U} = \Gamma$. Therefore, $\Upsilon_{X,U} (P,U)=\Gamma(P)$. 
Similarly $ \gamma _{X,\D} (Q) = (X,Q)$, for all $Q\in Ob(\D)$, and therefore $ \Upsilon_{X,U} (X,Q) = \Delta(Q) $,  since $\Upsilon_{X,U}\gamma_{X,\D}=\Delta$. 
Combining these, we see that 
\[ \Upsilon_{X,U}(X,U) \ = \ \Gamma(X)\ =\ \Delta(U)  \]
Finally, $X$ and $U$ were chosen arbitrarily, and so $\Gamma(X)\ =\ \Delta(U) $, for all $X\in Ob(\CC)$ and $U\in Ob(\D)$.
\end{proof}

\begin{remark}{The multi-object case, and universal arrows}
A simple corollary of the above proof is that, when such a universal family of functors exists,  $\Gamma: \CC\rightarrow \E$ and $\Delta: \D\rightarrow \E$ map all objects of $\CC$ and $\D$ to the same object of $\E$ (It is then straightforward, although deeply uninteresting, to demonstrate that in the restricted case where the target in the diagram of Figure \ref{epicfail} is a one-object PCM category, we do indeed have a universal property that is a direct generalisation of that of Theorem \ref{semiuniversal}). Note also that the above proof is entirely based on how such a universal family of  functors might act on the {\em objects} of these  categories. Therefore, it does not depend on any subtlety about the precise notion of summation used, or even the definition of composition; rather, it fails on the simple fact that $\E$ has more than one object.

From an algebraic point of view, this appears to be a serious drawback --- indeed, it is not uncommon to {\em define} the monoid semiring construction simply in terms of the existence of a suitable universal arrow. However, from a more computational point of view, it appears reasonable: from the interpretation given in Remark \ref{universalinterpretation} we should think of the construction of a universal arrow as substituting values for variables. When we consider the multi-object setting, the multiplicity of objects gives distinct {\em types} for both variables and values. The interpretation of Proposition \ref{notypeduniversals} is then that, when we try to substitute values for variables, we must do so in a context where all types agree. From a computational, rather than an algebraic, point of view this is to be expected!

However, this does not mean that no suitable universal property may exist. The key question is simply which object of $\E$ is  the appropriate target of $\Upsilon_{X,U}$?  Let us now assume that the PCM-category $\E$ also has a monoidal tensor $(\underline{\ }\otimes \underline{\ })$ that is required to satisfy some universal property for bilinearity, analogous to either that of the usual tensor product of Hilbert spaces, or the constructions of \cite{AB} for Partially Additive Monoids. In this case, the natural candidate must be the object $\Gamma (X) \otimes \Delta(U) \in Ob(\E)$. It therefore seems that questions of universal properties must await a theory of {\em monoidal} PCM-categories. As we demonstrate in Section \ref{neverending}, such a theory would be key to many reasonable generalisations and applications, both algebraic and computational.

\end{remark}

\section{Conclusions} 
We have demonstrated that the monoid-semiring construction can be placed within a significantly more general categorical and multi-object setting. In order for this more general theory to be equally applicable in both `analytic' and `algebraic' settings, this was done using  an axiomatisation of summation that unifies notions from analysis with notions from algebraic program semantics. 

\section{Future directions}\label{neverending}
As well as the program outlined above, work continues in several related directions:

\begin{itemize}
\item {\bf Categorical enrichment}
A natural question about this paper is whether the notion of `PCM-category' is in fact an example of categorical enrichment (as in \cite{GMK}) over the category {\bf PCM}? Enrichment requires either a monoidal, or a closed, (or monoidal closed) structure. It has recently been demonstrated by the author and P. Scott (Ottawa) that {\bf PCM} is a closed category in the sense of \cite{LA}, and a monoidal tensor adjoint to the closed structure has been given explicitly by T. Porter (Wales).  This monoidal tensor appears to exhibit a universal property for a suitable notion of bilinear maps of PCMs (similar to the constructions of \cite{AB} for Partially Additive Monoids). It is expected that a category enriched over this monoidal closed category is exactly a `PCM-category', as defined in Definition \ref{PCMcats}.  This is the subject of ongoing work.
\item {\bf The Cauchy product and monoidal structures}
Although we have demonstrated that the Cauchy product is a bifunctor, with interesting embedding properties, we have not yet considered the case where either the base category, or the index category (or both) have a monoidal tensor. This requires studying the monoidal structure of {\bf PCM} (as above) in order to describe what it means for a PCM-category to have a monoidal tensor. This is undoubtedly an interesting route to explore, and is also important for applications to semantics described below.
\item{\bf PCM-categories, algebraic program semantics, and axiomatisations of summation}
From the beginning, the notion of a PCM-category was intended as a unification of the forms of summation used in algebraic program semantics with more analytic notions of summation used in Hilbert and Banach spaces. A very natural question is therefore how much of the traditional theory may be carried through to this more general setting, and whether such constructions as the Elgot dagger, (presumably partial) particle-style categorical traces, etc. may be defined\footnote{As supporting evidence that this is possible, we refer to \cite{HiSc,TCS2} where partial categorical traces, based on notions of summation that do not satisfy positivity (but do, however, satisfy the PCM axioms), are both used to model quantum-optics thought experiments and to construct concrete quantum circuits.}.  Again, the absolute starting point for this is the definition of suitable monoidal tensors, and their interaction with notions of summation. 
\end{itemize}
Note that all the above future directions depend on a detailed study of monoidal tensors and closed structures in both $\PCM$ and members of $\CS$. Thus, it appears that pursuing such questions as pure theory may be the most profitable route!

\section*{Acknowledgements}
The author wishes to acknowledge the contribution of P. Scott (Ottawa) to the definitions and properties of Partial Commutative Monoids, as developed from Section \ref{PCM-start} to Section \ref{PCM-end}, as well as many of the examples presented in Appendix \ref{appendix}. Thanks are also due to S. Abramsky (Oxford) for insisting on the importance of permutation independence (Proposition \ref{welldefined}), and pointers towards the proof of this. Similarly, thanks are due to T. Porter (Wales) for assistance with the monoidal closed structure of the category $\PCM$ (an ongoing project). 

The author also wishes to thank anonymous referees of MSCS for many helpful comments and suggestions that have assisted substantially in refining and clarifying the theory developed, and the editor M. Mislove (Tulane) for a great deal of patience!

\appendix

\section{PCMs and PCM-categories}\label{appendix}
We consider various examples of both PCMs and PCM-categories, as defined in Definitions \ref{PCM} and \ref{PCMcats} respectively. We also compare with other axiomatisations of summation from the field of algebraic program semantics.

\subsection{Examples of PCMs from algebraic program semantics}

Both {\em $\Sigma$-monoids}, and {\em partially additive monoids}, as introduced in \cite{EMDB,MA} and used in \cite{HA,AHS,HS}, may be given as special cases of PCMs: 
\begin{definition}\label{sigmamon}{\em ($\Sigma$-monoids, Partially additive monoids)}\\ 
A PCM $(M,\Sigma )$ is called a 
{\bf $\Sigma$-monoid} when it satisfies the following additional axiom: 
\begin{itemize}
\item The {\bf (full) Partition-Associativity Axiom.} Let $\{x_i\}_{i\in I}$ be a countably indexed
family, and let $\{I_j\}_{j\in J}$ be a countable partition of $I$. Then $\{ x_i\}_{i\in
I}$ is summable if and only if $\{ x_i\}_{i\in I_j}$ is summable for every $j\in J$, and
$\{ \sum_{i\in I_j}x_i \}_{j\in J}$ is summable, in which case
\[ \sum_{i\in I} x_i = \sum_{j\in J} \left( \sum_{i\in I_j} x_i \right) \]
\end{itemize}
Note that this is a special case of the {\em weak partition-associativity axiom}, with a two-way, instead of a one-way, implication. 

A $\Sigma$-monoid is called a {\bf Partially Additive Monoid} (PAM) when it satisfies the following additional axiom:
\begin{itemize}
\item The {\bf Limit Axiom.} Given $\{ x_i\}_{i\in I}$, a countably indexed family where$\{
x_i\}_{i\in F}$ is summable for every finite
$F\subseteq I$, then $\{x_i\}_{i\in I}$ is summable.
\end{itemize}  
\end{definition}

\noindent
The following are Partially Additive Monoids, and are therefore examples of PCMs:
\begin{itemize}
\item {\em Partial functions, with the usual summation}\\
An indexed family of partial functions $\{ f_i : X\rightarrow Y \}_{i\in I}$  is {\bf summable} exactly when $dom(f_i)\cap dom (f_j) = \emptyset$ for all $i\neq j$. The {\bf sum} is given by: \\ 
$ \left( \sum_{i\in I} f_i \right) (x) = \left\{ \begin{array}{lr} f_{i}(x) & x\in dom(f_i) \\ \mbox{undefined} & \mbox{otherwise} \end{array}\right.$
\item {\em Relations, with set-theoretic union}\\
Any indexed family of relations $\{ R_i :X\rightarrow Y \}_{i\in I}$ is {\bf summable}, and the {\bf sum} is simply set-theoretic union.
\item {\em Partial injective functions}\\
The following distinct summations both give a PAM structure to hom-sets of partial injective functions:
\begin{itemize}
\item {\em The disjointness summation}
An indexed family of partial functions $\{ f_i : X\rightarrow Y \}_{i\in I}$  is {\bf disjointness-summable} exactly when $dom(f_i)\cap dom (f_j) = \emptyset$ for all $i\neq j$. The {\bf sum} is given by: \\ 
$ \left( \sum_{i\in I} f_i \right) (x) = \left\{ \begin{array}{lr} f_{i}(x) & x\in dom(f_i) \\ \mbox{undefined} & \mbox{otherwise} \end{array}\right.$
\item {\em The overlap summation}
An indexed family of partial functions $\{ f_i : X\rightarrow Y \}_{i\in I}$  is {\bf overlap-summable} exactly when $x\in dom(f_i)\cap dom (f_j) \ \Rightarrow \ f_i(x) =f_j(x)$, for all $i,j\in I$. The {\bf sum} is as given above.
\end{itemize}
\end{itemize}
The following example is not a partially additive monoid, but is a $\Sigma$-monoid, and thus also an example of  a PCM:
\begin{itemize}
\item {\em Absolute convergence on positive cones}\\ 
We refer to \cite{PS} for categories of positive cones, and summation based on the usual summation of positive elements in finite-dimensional vector space.
\end{itemize}

\noindent
Given our stated aim of unifying notions of summation from both analysis and algebraic program semantics, both $\Sigma$-monoids, and Partially Additive Monoids have undesirable properties for our purposes. The {\em limit axiom} is clearly undesirable for any example based on real or complex numbers: all finite families of complex numbers  are summable, but the same is certainly not true (as the limit axiom would imply) for arbitrary countably infinite families.

The full partition-associativity axiom is also undesirable for slightly more subtle reasons, as the following proposition (taken from \cite{MA}) demonstrates: 
\begin{proposition}\label{positive}
Let $(M ,\Sigma )$ be a $\Sigma$-monoid, and let $X=\{ x_i\}_{i\in I}$ be a summable family of $M$ satisfying $\sum_{i\in I} x_i=0$. Then $x_i = 0$ for all $i\in I$.
\end{proposition}
\begin{proof}
For some $i\in I$, we define $Y=\{ x_j\}_{j\neq i\in I}$, so $x_i + \sum Y = 0 = \sum Y + x_i$ by weak partition associativity. Then by the full partition-associativity axiom, 
\begin{eqnarray*}
 x_i & = & x_i + 0 + 0 + 0 + \ldots  \quad \mbox{by 
 Proposition \ref{pcmsums}}\\
  & = & x_i + (\sum Y + x_i) + (\sum Y + x_i) + (\sum Y + x_i) + \ldots \\
  & & ( \mbox{by full partition
 associativity})\\
 & = & (x_i + \sum Y) + (x_i + \sum Y) + (x_i + \sum Y) + \ldots \\
 & = & 0 + 0 + \cdots ~ = ~ 0 , \mbox{by 
 Proposition \ref{pcmsums}} \end{eqnarray*}
 Hence $x_i=0$. However, as $i$ was chosen arbitrarily,
$x_k=0$ for all $k\in I$. 
\end{proof}

\subsection{Positivity, computation, and the PCM axioms}
From a certain point of view, positivity seems to be a natural property of notions of summation used in theoretical computer science. Taking the `sum' of a family of arrows in a category is often interpreted in a very domain-theoretic manner, as looking at the total information provided by all these arrows. The challenge to this intuition comes from the field of quantum computation, where summing amplitudes leads to both constructive and destructive interference effects. From \cite{DEL}

\begin{quotation} Amplitudes are complex numbers and may cancel each other, which is referred to as destructive interference, or enhance each other, referred to as constructive interference. The basic idea of quantum computation is to use quantum interference to amplify the correct outcomes and to suppress the incorrect outcomes of computations.
\end{quotation}

From this point of view, at least, enforcing positivity in models of quantum computation would seem to rule out the phenomena that distinguish quantum computation. A good example is provided by the quantum Fourier transform  (required in, for example, Shor's algorithm \cite{Shor} and quantum period-finding generally \cite{NC}), which is based on group homomorphisms $\chi:{\Bbb Z}_n \rightarrow \mbox{\bf Hilb}(H,H)$ satisfying $\sum_{j=0}^{n-1}\chi(j)= 0_H$. Clearly, assuming positivity will only allow for the trivial homomorphism.

Finally, we refer to \cite{TCS2} for an application of category theory to quantum circuits that relies on both summing linear maps, and composition based on convolved (i.e. Cauchy) products.

\subsection{Non-positive examples of PCMs}\label{analyticexamples}
The proof of positivity for Sigma monoids given in Proposition \ref{positive} does {\em not} apply to general PCMs, as it depends on the two-way implication in the (full) partition-associativity axiom. We give various examples of PCMs that need not be either Partial Additive Monoids or Sigma-monoids. Many of these are based on the theory of Cauchy sequences  \cite{EH,T}, and various analytic notions of summability, such as absolute convergence of real or complex sums. 

\begin{definition}\label{realsums} Let $ \sum_{j=0}^\infty 
 a_j$ be a formal (i. e. not necessarily convergent) series of real numbers. The $n^{th}$ partial sum is  
defined by $A_n=\sum_{j=0}^n a_j$. When $lim_{n\rightarrow \infty}(A_n)$ exists, then the infinite series is said to {\bf converge}. Note that convergence is {\em permutation-dependent}. Let $ \sum_{j=0}^\infty 
 a_j$ and $\sum_{k=0}^\infty b_k$, be series satisfying $b_k=a_{\sigma(j)}$, for some  permutation $\sigma: \Bbb N \rightarrow \Bbb N$. Then the convergence of $ \sum_{j=0}^\infty 
 a_j$ is not enough to guarantee the convergence of  $\sum_{k=0}^\infty b_k$. 
 
 A convergent series $ \sum_{j=0}^\infty 
 a_j$ is said to {\bf converge absolutely} when it satisfies the additional property that the non-negative series $ \sum_{j=0}^\infty |
 a_j|$ is convergent.  Alternatively, a convergent series $ \sum_{j=0}^\infty 
 a_j$  is said to be {\bf permutation-independent} when $\sum_{i=0}^\infty a_{\sigma(i)}$ converges for arbitrary permutations $\sigma:{\Bbb N} \rightarrow {\Bbb N}$. 
 \end{definition}
 The following is straightforward:
 \begin{lemma}The real line, with summation defined by convergence, is {\em not} a PCM.\end{lemma}\begin{proof} From Proposition \ref{welldefined}, summation in a PCM must satisfy permutation independence.\end{proof}
 
 A standard result of analysis is that for real and complex numbers, absolute convergence is equivalent to permutation-independence\footnote{We emphasise that this is specific to the real and complex planes, and finite-dimensional spaces. In more general settings,  general these are distinct concepts. In particular, infinite-dimensional Banach spaces or abstract topological groups provide counterexamples, as discussed following Definition \ref{banachsums}.}.
Real numbers with summation defined by absolute convergence then provides our first example of a PCM not satisfying the positivity property.

\begin{proposition}
The real number line $\Bbb R$, together with summation defined by absolute convergence, satisfies the axioms for a PCM.
\end{proposition}
\begin{proof} It is a triviality that the unary sum axiom is satisfied. To see that the weak partition associativity axiom is also satisfied, we refer to \cite{EiHi}, p. 108 (also quoted in Remark \ref{hilleQuote} of this paper). See also Theorem 8 of \cite{KK}, p. 84.
\end{proof}

The following corollary is then immediate:
\begin{corol}
the complex plane $\Bbb C$, together with the summation defined by absolute convergence, satisfies the PCM axioms.
\end{corol}  
 
The above results may be extended to finite-dimensional Hilbert and Banach spaces, with no substantial obstacles.  However, it is more satisfactory to consider summation in the general setting, and restrict to these as special cases. As a preliminary, we need additional analytic notions of summation. The following definition is taken from \cite{MD}:

\begin{definition}\label{banachsums} Let $X$ be an arbitrary Banach space. A series $\sum_{i=0}^\infty x_i$ is said to {\bf converge absolutely} when  $\sum_{i=0}^\infty \| x_i\| <\infty$. Alternatively, is is said to {\bf converge unconditionally} when the series $\sum_{j=0}^\infty x_{\sigma (j)}$ converges for arbitrary permutations $\sigma : {\Bbb N}\rightarrow {\Bbb N}$. (Unconditional summability can simply be thought of as permutation-independent convergence, in a more general setting.) It is standard that for any unconditionally convergent series, all rearrangements have the same sum;  also, every subseries of an unconditionally convergent series is itself unconditionally convergent \cite{BT}.
\end{definition}

\begin{theorem}\label{finitedim}
In an arbitrary Banach space, absolute convergence implies unconditional convergence, but the converse is not generally true. However, in finite-dimensional Banach spaces, absolute convergence and unconditional convergence are equivalent.
\end{theorem}
\begin{proof} This is a standard result of analysis -- see, for example,  \cite{KK}, Theorem 1.3.3
\end{proof}

\begin{definition}\label{subseries}
Let $X$ be an arbitrary Banach space. A series $\sum_{i=0}^\infty a_i$ is {\bf subseries convergent} when the partial sums $A_n = \sum _{j=0}^n \alpha(j) a_j$ form a Cauchy sequence, for arbitrary choice of $\alpha : \Bbb N \rightarrow \{ 0,1 \}$. Subseries convergence is often defined informally as `each subseries of $\sum_{i=0}^\infty a_i$ converges'.
\end{definition}
In Banach spaces, subseries convergence provides a nice characterisation of unconditional convergence, as the following classic theorem demonstrates:
\begin{proposition}
Let $X$ be an arbitrary Banach space. A series $\sum_{i=0}^\infty x_i$ is unconditionally convergent if and only if it is subseries convergent.
\end{proposition}
\begin{proof}
This is a corollary of the classic result of  \cite{WO}. We refer to \cite{MD} for a (English language) textbook proof, and \cite{BLPD} for a nice elementary proof.\end{proof}  

Note that the above result depends on the sequential completeness of Banach spaces, and thus in more general settings, subseries-convergence and unconditional convergence are not equivalent concepts. In particular, conditions equivalent to\footnote{We refer to \cite{CMA} for many conditions equivalent to subseries-convergence, including the definitions of \cite{IG}.} subseries convergence were introduced in \cite{IG} as `strong unconditional convergence'.

\begin{corol} Absolutely convergent series in finite-dimensional Banach spaces are subseries-convergent.\end{corol}
\begin{proof}This follows from Theorem \ref{finitedim} above.\end{proof}
As may be expected, subseries-convergent series satisfy the weak partition associativity axiom. To prove this, we first need some unsurprising technical results.
\begin{definition} Let $\sum_{j=0}^\infty y_j$ be an infinite series, and $\sum_{i=0}^\alpha x_i$ be a finite or infinite series in some Banach space $X$. We say that  $\sum_{j=0}^\infty y_j$ is a 0-padding of $\sum_{i=0}^\alpha x_i$ when there exists some injection $\eta:\{0,\ldots \alpha\} \rightarrow {\Bbb N}$ such that for all $j\in {\Bbb N}$, 
\[ y_j =\left\{\begin{array}{lr} x(i) & j=\eta(i) \\ 0 & \mbox{otherwise.}\end{array}\right. \]
\end{definition}

\begin{lemma}\label{technical}Let  $\sum_{i=0}^\alpha x_i$ be a finite or countably infinite subseries-summable series in some Banach space $X$  and let $\sum_{j=0}^\infty y_j$ be a zero-padding of $\sum_{i=0}^\infty x_i$. Then $\sum_{j=0}^\infty y_j$ is subseries convergent and $\sum_{j=0}^\infty y_j = \sum_{i=0}^\alpha x_i$.
\end{lemma}
\begin{proof} Let us assume that $\sum_{i=0}^\infty x_i$ is an infinite series, otherwise the result is trivial.  Consider arbitrary  $\beta : \Bbb N \rightarrow \{ 0,1 \}$, along with the sum   $\sum_{j=0}^\infty \beta(j)y_j$, and let $\eta:{\Bbb N}\rightarrow {\Bbb N}$ be the embedding satisfying $y_j =\left\{\begin{array}{lr} x(i) & j=\eta(i) \\ 0 & \mbox{otherwise.}\end{array}\right. $. Then since the partial sums $\left\{ \sum_{i=0}^n x_i\right\}_{n\in {\Bbb N}}$ form a Cauchy sequence so do the partial sums $\left\{ \sum_{j=0}^m y_j\right\}_{m\in {\Bbb N}}$. Thus, as $\beta : \Bbb N \rightarrow \{ 0,1 \}$ was chosed arbitrarily, we deduce that $\sum_{j=0}^\infty y_j$ is subseries convergent, as required. The equivalence of the two sums is then immediate.
\end{proof}

This technicality then allows us to appeal to standard results, in order to demonstrate that subseries-convergent series in Banach spaces satisfy the Weak Partition Associativity axiom.
\begin{theorem} Let $\sum_{i=0}^\infty x_i$ be a subseries-convergent series in an arbitrary Banach space $X$, and let 
$\{I_j\}_{j\in J}$ be a countable partition of $I$. Then 
$\{ x_i\}_{i\in I_j}$ is subseries-convergent, as is $\{ \sum_{i\in I_j}x_i
\}_{j\in J}$, and $ \sum_{i\in I} x_i = \sum_{j\in J} \left(\sum_{i\in I_j} x_i \right)$.
\end{theorem}
\begin{proof} This result is proved, for partitions into countably infinite sets, in Lemma 2 of \cite{BT}.  The case where certain of these partitions are finite appears to be implicitly assumed in \cite{BT} -- for a formal justification, we may consider zero-padding the original sequence, to replace finite subsums by infinite sums with a finite number of non-zero summands, and appealing to Lemma \ref{technical} above.
\end{proof}

\noindent
We may now list a number of PCMs that do not, or are not required to, satisfy the positivity condition. 
\begin{enumerate}
\item {\em Absolute convergence of real or complex numbers}  Absolute convergence of countable sums in the real or complex plane is a motivating example for the theory of PCMs and PCM-categories. It arises as a special case of absolute convergence in finite-dimensional Hilbert spaces, as below:
\item {\em Absolute convergence in finite-dimensional Hilbert spaces} All finitely indexed families are summable, with the usual summation. A countably indexed family $\{ \psi_i \}_{i\in {\Bbb N}}$ is {\bf summable} exactly when the sequence $\{ \sum_{i=1}^n \| \psi_i \| \}_{n\in {\Bbb N}}$ is a Cauchy sequence.
\item {\em Subseries-summable summation in arbitrary Banach spaces} All finitely indexed families are summable, with the usual summation. A countably indexed family $\{ b_i \}_{i\in {\Bbb N}}$ is {\bf summable} exactly when the series $\sum_{i=0}^\infty b_i$ is subseries-summable in the sense of Definition \ref{subseries}, in which case the sum is the limit of the Cauchy sequence $\{ \sum_{i=1}^n b_i  \}_{n\in {\Bbb N}}$.
\item {\em The unit ball summation in finite-dimensional Banach spaces} Let ${\mathcal B}$ be a finite-dimensional  Banach space, and denote the unit ball by $Ball({\mathcal B}) =\{ b\in {\mathcal B}: \| b \| \leq 1\}$. An indexed family $\{ b_i \}_{i\in I}$ is summable exactly when $\sum_{i\in I} \| b_i \|\leq 1$, in which case its sum is the usual Banach space summation.
\item {\em Any abelian monoid} Given an abelian monoid $(M,+,0_M)$, then the following are distinct PCM-structures: 
\begin{itemize}
\item {\em The finite families summation}  An indexed set $\{  m_i\}_{i\in I}$ is summable exactly when the subfamily of non-zero elements $\{ m_j \}_{j\in J\subseteq I}$ is a finite family. The sum is defined by 
\[ \sum_{i\in I} m_i = \left\{ \begin{array}{clr} \sum_{j\in J} m_j & J \neq \{ \} \\ & \\ 0_M & \mbox{otherwise,} \end{array}\right. \]
where the finitary sum on the r.h.s. is the usual sum of the abelian group.
\item {\em The $K$-bounded summation} As above, but where the summable families are those with at most  $K$ non-zero elements.
\end{itemize}
It is almost immediate from the commutativity and associativity of composition in $M$ that these both satisfy the PCM axioms.
\item {\em Any abelian group}, with either of the above the finite families summation, or $K$-bounded summation, as a special case of the abelian monoid examples.
\end{enumerate}

\subsection{Examples of PCM-categories}
A PCM-category is defined in Definition \ref{PCMcats} to be a category $\CC$ where each hom-sets has a specified PCM structure, together with the {\em strong distributivity} axiom that connects composition and summation. This states that, given summable families $\{ g_j \in \CC(Y,Z)\}_{j\in I}$ and $\{ f_i \in \CC(X,Y) \}_{i\in I}$, then $\{ g_jf_i\in \CC(X,Z) \}_{(j,i)\in \CC(X,Z)}$ is summable and 
\[ \left( \sum_{j\in J} g_j \right) \left( \sum_{i\in I} f_i \right) \ = \ \sum_{(j,i)\in J\times I } g_jf_i \]

\begin{itemize}
\item {\em The real or numbers, with multiplication and absolute convergence}\\ 
 Given two absolutely convergent sums of real or complex numbers, $\sum_{i\in I} r_i $ and $\sum_{j\in J} s_j$, then by definition of absolute convergence, $\sum_{(i,j)\in I\times J} r_i s_j$ exists, and 
 \[ \left( \sum_{i\in I} r_i \right) \left( \sum_{j\in J} s_j\right) =\sum_{(i,j)\in I\times J} r_is_j \] 
 Therefore, $({\Bbb R},\times )$ or $({\Bbb C},\times)$, with this indexed summation, is a one-object PCM-category.
 \item {\em Linear maps on finite-dimensional Hilbert space, with composition and uniform convergence} \\
 This follows similarly to the above examples. Note that the hom-set of maps between two finite-dimensional Hilbert spaces is itself a finite-dimensional Hilbert space, and the distinct notions of convergence with respect to various operator norms all coincide in the finite-dimensional case. The strong distributivity property is a classic result of analysis (see, for example \cite{CS} for a general setting).
\item {\em Any ring, with the finite families summation}\\ 
 Let $(R,\times,+)$ be a ring. Then $(R,\times )$ is a monoid, and hence a one-object category. Its unique homset (i.e. the elements of $R$) is an abelian monoid, $(R,+)$. Hence we may take the the finite families summation $\Sigma_{<\infty}$ of Section \ref{analyticexamples} above to get the PCM $(R,\Sigma_{<\infty})$. Given two summable families, $\{ s_j \}_{j\in J}$ and $\{ r_i\}_{i\in I}$, then the family $\{ s_j r_i \}_{(j,i)\in J\times I}$ has a finite number of non-zero elements, and hence is summable. Given the summability of the required families, the identity 
\[ \left( \sum_{j\in I} s_j \right) \left( \sum _{i\in I} r_i \right) = \sum_{(j,i)\in J\times I} s_jr_i \]
is then straightforward from the definition of summation in terms of the addition in the ring $(R,\times ,+)$.
\end{itemize}

Note that, simply because the homsets of a category are PCMs, we no not necessarily have a PCM-category --- we also need the strong distributivity condition of Definition \ref{PCMcats}. For example, consider a unital ring $(R,\times,+,1,0)$. The multiplicative monoid $(R,\times,1)$  is trivially a one-object category, and as demonstrated in Section \ref{analyticexamples}, we may give the additive abelian monoid $(R,+,0)$  a PCM structure using the  K-bounded summations $\Sigma_{\leq K}$ where a family is summable exactly when it has no more than $K$ non-zero elements. However, for $K>1$, the K-bounded summation does {\em not} in general make $(R,\times , \Sigma_{\leq K})$ a one-object PCM category. Let us assume that $R$ has no zero-divisors, and consider two summable families containing $K$ non-zero elements $\{ s_j \}_{j\in J}$ and $\{ r_i \}_{i\in I}$. Then the strong distributivity law does not hold, since $\{ s_jr_i \}_{(j,i)\in J\times I}$ is not a summable family, as it contains $K^2>K$ non-zero elements.\\

We now demonstrate that there is a whole class of examples to be found within the field of algebraic program semantics. The proofs that these are PCM-categories arises from the following straightforward result:
\begin{proposition}\label{sdfromfpa}
Let $\CC$ be a category, together with, for all $X,Y\in Ob(\CC)$ a function $\Sigma^{(X,Y)}$ from indexed families over $\CC(X,Y)$ to $\CC(X,Y)$ such that:
\begin{enumerate}
\item $\left( \CC(X,Y), \Sigma^{(X,Y)} \right)$ is a PCM satisfying the additional {\em full partition-associativity axiom} of Definition \ref{sigmamon}. 
\item  The usual left and right distributivity conditions as satisfied:  that is, given a summable family $\{ g_i\in \CC(B,C) \}_{i\in I}$ and arbitrary $f\in \CC(A,B)$ and $h\in \CC(C,D)$, then the families $\{ hg_i\in \CC(B,D) \}_{i\in I}$ and $\{ g_if\in \CC(A,C) \}_{i\in I}$ are summable and 
\[ h\left( \sum_{i\in I} g_i \right) \ =\  \sum_{i\in I} (hg_i) \ \ \mbox{ and } \ \ \left( \sum_{i\in I} g_i \right)   f \ = \ \sum_{i\in I} (g_if) \] 
\end{enumerate}
then $\left( \CC,\Sigma^{(\underline{\ },\underline{\ })}\right)$ satisfies the strong distributivity condition of Definition \ref{PCMcats}, and thus is a PCM-category.
\end{proposition}
\begin{proof}
Let us write $f=\sum_{i\in I} f_i$ and $g=\sum_{j\in J} g_j$. Then, 
By left-distributivity, and the existence of $\sum_{i\in I} f_i$, we deduce $gf=\sum_{i\in I} gf_i$. However, $g=\sum_{j\in J} g_j$. Therefore, by full partition-associativity, $gf = \sum_{i\in I} \left( \sum_{j\in J} g_jf_i \right)$. 
Using the right distributivity law and full partition-associativity, $gf = \sum_{j\in J} \left( \sum_{i\in I} g_jf_i \right)$. 
Again by full partition-associativity, and Proposition \ref{welldefined}, we may replace the doubly-indexed sum by a single indexed sum, giving  $gf=\sum_{(j,i)\in J\times I} g_jf_i$ and hence
\[ gf \ = \ \left( \sum_{j\in J}g_j \right) \left( \sum_{i\in I} f_i \right) \  = \  \sum_{(j,i)\in J\times I} g_jf_i \]
Therefore $\CC$ satisfies strong distributivity, as required.
\end{proof}

\noindent Note that the converse is not true: weak partition-associativity, together with the strong distributivity law, does not, in general, imply the full partition-associativity axiom. This is clear from the failure of positivity in many of the examples given.

 \begin{corol}The Partially Additive Categories (PACs) of \cite{MA} are PCM-categories, as are the Unique Decomposition Categories (UDCs) of  \cite{HA}, \cite{AHS}, \cite{HS}.
 \end{corol}


\begin{thebibliography}{5}
\bibitem[Abramsky et. al. 2002]{AHS}{S. Abramsky, E. Haghverdi, P. Scott (2002)} Geometry of interaction and linear combinatory algebras, {\em Math. Struct. Comput. Sci. 12 (5)}  625-665 
\bibitem[Abramsky 2005]{SA} {S. Abramsky (2005)} Abstract Scalars, Loops, and Free Traced and Strongly Compact Closed Categories, {\em Springer LNCS Vol. 3629} 1- 29
\bibitem[Bahamonde 1985]{AB} {A. Bahamonde (1985)} Tensor Product of Partially-Additive Monoids, {\em Semigroup Forum (32)} 31-53 
\bibitem[Day 1973]{MD} {M. Day} {\em Normed Linear Spaces} ($3^{rd}$ edition), Springer-Verlag, New York
\bibitem[Deutsch et. al 1999]{DEL}{D. Deutsch, A. Ekert, R. Lupacchini (1999)} Machines, Logic, and Quantum Physics {\em arXiv:math.HO/9911150 v1}
\bibitem[Haghverdi 2000]{HA} {E. Haghverdi (2000)} A categorical approach to linear
logic, geometry of proofs and full completeness, {\em PhD Thesis, Univ. Ottawa}
\bibitem[Haghverdi, Scott 2006]{HS}{E. Haghverdi, P. Scott (2006)} A categorical model for the Geometry of Interaction, {\em Theoretical Computer Science 350(2)} 252-274
\bibitem[Gelfand 1938]{IG} I. Gelfand (1938)  Abstrakte Funktionen und lineare Operatoren {\em Rec. Math. [Mat. Sbornik] N.S.  4(46) vol. 2} 235-286 
\bibitem[Golan 1999]{JG}{J. Golan (1999)} {\em Power Algebras Over Semirings: With Applications in Mathematics and Computer Science}, 
Springer, Mathematics and Its Applications series, Vol. 488
\bibitem[Hille 1982]{EiHi}{Einar Hille (1982)} {\em Analytic Function Theory, Vol. 1 (2nd ed.)} AMS Chelsea publishing
\bibitem[Hines 2008a]{TCS}{P. Hines (2008)(a)} Machine Semantics, {\em Theoretical Computer Science 409(1)} 1-23
\bibitem[Hines 2008b]{IJUC}{P. Hines (2008)(b)} Machine Semantics: from Causality to Computational Models, {\em International Journal of Unconventional Computation 4(3)} 249-272
\bibitem[Hines 2010]{TCS2}{P. Hines (2010)} Quantum Circuit Oracles for Abstract Machine Computations {\em Theoretical Computer Science  411} 1501-1520
\bibitem[Hines, Scott]{HiSc}{P. Hines, P. Scott (2012)} Categorical Traces from Single-Photon Linear Optics {\em in S. Abramsky, M. Mislove (ed.s), AMS Proceedings of Symposia in Applied Mathematics (71)} 89-124
\bibitem[Hobson 1957]{EH}{E. W. Hobson (1957)} {\em The Theory of Functions of a Real Variable \& the Theory of Fourier's Series, Vol.1, $2^{nd}$ Edition}, Dover Publications, New York
\bibitem[Kadets, Kadets 1991]{KK}{ V. Kadets, M. Kadets 1991} {\em Rearrangements of Series in Banach Spaces} Am. Math. Soc., Providence
\bibitem[Kelly 1982]{GMK} {G. M. Kelly (1982)} Basic Concepts of Enriched Category Theory, {\em LMS Lecture notes 64}, Cambridge University Press, {\em Reprinted in \cite{GMK2}}
\bibitem[Kelly 2005]{GMK2} {G. M. Kelly (2005)} Basic Concepts of Enriched Category Theory, {\em Reprints in Theory and Applications of Categories (10)} 
\bibitem[McArthur 1961]{CMA}{C. McArthur (1961)} A note on Subseries Convergence {\em Proc. Am. Math. Soc. (12) 4}  540-545
\bibitem[Manes, Arbib 1986]{MA}  {E. Manes, M. Arbib (1986)} {\em Algebraic Approaches to
Program Semantics}  Springer-Verlag
\bibitem[Manes, Benson 1985]{EMDB}  {E. Manes, D. Benson (1985)} The inverse Semigroup of
a Sum-Ordered Semiring {\em Semigroup Forum 31} 129-152
\bibitem[Lahiri, Das 2002]{BLPD}{Benoy Kumar Lahiri, Pratulananda Das (2002)} Subseries in Banach spaces {\em Mathematica Slovaca (52) 3} 361-368
\bibitem[Laplaza 1977]{LA}{M. L. Laplaza (1977)} Coherence in Nonmonoidal Closed Categories, {\em Trans. Am. Math. Soc.}, V. 230  293-311
\bibitem[Nielsen, Chuang 2000]{NC}  {M. Nielsen, I. Chuang (2000)} {\em Quantum Computation and Quantum Information}, Cambridge University Press
\bibitem[Orlicz 1933]{WO} {W. Orlicz (1933)} \"Uber unbedingte Konvergenz in Functionenraumen I, {\em Studia Math. 4} 33-37
\bibitem[Swartz 1992]{CS}{C. Swartz (1992)} Iterated series and the Hellinger-Toeplitz theorem {\em Publicacions Matem\`atiques 36} 167-173.
\bibitem[Selinger 2004]{PS}{P. Selinger (2004)} Towards a quantum programming language, {\em Mathematical Structures in Computer Science 14(4)} 527-586
\bibitem[Shor 1999]{Shor}{P. Shor (1999)} Polynomial time algorithms for prime factorisation and discrete logarithms on a quantum computer {\em SIAM review 41} 303-332
\bibitem[Steinberger 1993]{MS}{M. Steinberger (1993)} {\em Algebra} Prindle, Weber and Schmidt {\em updated version (2006) available online as http://math.albany.edu/$\sim$mark/algebra.pdf}
\bibitem[Titchmarsh 1983]{T} {E. C.Titchmarsh (1983)} {\em The Theory of Functions, $2^{nd}$ Edition}, Oxford University Press
\bibitem[Thorpe 1968]{BT} {B. Thorpe (1968)} On the equivalence of certain notions of bounded variation {\em Journal London  Math. Soc. 43} 247-252
\end{thebibliography}
\end{document}